\documentclass[a4paper,12pt]{article}

\usepackage[english]{babel}
\usepackage[utf8]{inputenc}
\usepackage{etex}
\usepackage{pdfpages}
\usepackage{amsmath}
\usepackage{amssymb}
\usepackage{amsthm}
\usepackage{appendix}
\usepackage{comment}
\usepackage{hyperref}
\hypersetup{colorlinks,linkcolor={blue},citecolor={blue},urlcolor={red}} 
\usepackage{cases}
\usepackage{enumerate}
\usepackage{caption}
\usepackage{setspace}
\usepackage{ulem}
\usepackage{enumitem}
\usepackage{scalerel}

\usepackage{upgreek}

\usepackage{manfnt} 
\usepackage{mathrsfs}
\usepackage{minitoc}
\usepackage[natbibapa]{apacite}
\usepackage{accents}

\newcommand{\q}[1]{``#1''}
\usepackage{amsmath,amsfonts,bbm,amsthm}
\usepackage{cleveref}       
\usepackage{amssymb}
\usepackage{float}
\usepackage{latexsym}
\usepackage{eucal}
\usepackage{mathtools}
\usepackage{tikz}
\usetikzlibrary{decorations.markings}
\usepackage{pgfplots}
\usepackage{lipsum}
\usepackage{color}
\usepackage{tabularx}
\usepackage{array}
\usepackage[linesnumbered,ruled,vlined]{algorithm2e}
\SetKwInput{KwInput}{Input}
\SetKwInput{KwOutput}{Output}
\SetKwInput{KwEnsure}{Ensure}
\usepackage{graphicx,caption,subcaption}
\usepackage{moreverb}
\usepackage{multirow} 
\usepackage{url} 
\usepackage[all]{xy} 
\usepackage{shorttoc} 
\usepackage{textcomp} 
\usepackage{hyperref} 
\usepackage{url}
\usepackage{enumerate}
\usepackage{enumitem}
\usepackage[right]{eurosym}
\def \R{\mathbb{R}}

\def \N{\mathbb{N}}
\def\interior{\mathrm{int}}

\usepackage{graphics}
\usepackage{graphicx}
\usepackage{hyperref}

\usepackage{lastpage}
\usepackage{latexsym}

\usepackage{color}
\usepackage{comment}

\usepackage{enumerate}

\usepackage{fancyhdr}
\usepackage[top=3cm, bottom=3cm, left=1.5cm , right=2cm]{geometry}
\newtheorem{Theorem}{Theorem}[section]

\newtheorem{Lemma}[Theorem]{Lemma}
\newtheorem{Definition}[Theorem]{Definition}
\newtheorem{Proposition}[Theorem]{Proposition}
\newtheorem{Remark}[Theorem]{Remark}

\newtheorem{Assumption}[Theorem]{Assumption}

\newtheorem{Example}{Example}

\numberwithin{equation}{section}

\usepackage[framemethod=tikz]{mdframed}
\usetikzlibrary{decorations.pathmorphing,calc}
\newcommand\wavydecor{%
    \draw[decoration={coil,aspect=0.1,segment length=5pt,amplitude=1.0pt},decorate,line width=1.5pt,black]
      (O|-P) -- (O);
}
\newmdenv[
hidealllines=true,
innerleftmargin=10pt,
innerrightmargin=0pt,
innertopmargin=0pt,
innerbottommargin=0pt,
leftmargin=-10pt,
skipabove=.5\baselineskip,
skipbelow=.5\baselineskip,
singleextra={\wavydecor},
firstextra={\wavydecor},
secondextra={\wavydecor},
middleextra={\wavydecor}
]{done}

\graphicspath{{../figures/}}
\usepackage[linecolor=black,textwidth=25mm,textsize= footnotesize]{todonotes} 

{ \begin{list}%
	{$\bullet$}%
	{\setlength{\labelwidth}{10pt}%
	 \setlength{\leftmargin}{35pt}%
	 \setlength{\itemsep}{\parsep}{1pt}}}%
{ \end{list} }

\title{ Estimation of Systemic Shortfall Risk Measure using Stochastic Algorithms\thanks{Acknowledgements:  The authors research is part of the ANR project DREAMeS (ANR-21-CE46-0002) and benefited from the support of the "Chair Risques Emergents en Assurance" under the aegis of Fondation du Risque, a joint initiative by Le Mans University and Cov\'ea.}}  

\author{Sarah Kaaka\"i~ \thanks
{\small Laboratoire Manceau de Math\'ematiques \& FR CNRS N\textsuperscript{o} 2962, Institut du Risque et de l'Assurance, Le Mans University.}
\and Anis Matoussi
\thanks{ \small  Laboratoire Manceau de Math\'ematiques \& FR CNRS N\textsuperscript{o} 2962, Institut du Risque et de l'Assurance, Le Mans University and Ecole Poyltechnique.
}\\
 \and Achraf Tamtalini~ \thanks
{\small  Laboratoire Manceau de Math\'ematiques \& FR CNRS N\textsuperscript{o} 2962, Institut du Risque et de l'Assurance, Le Mans University.}
}
\date{}

\begin{document}
 \maketitle

  \abstract{   Systemic risk measures were introduced to capture the global risk  and the corresponding contagion effects that
are  generated by an interconnected system of financial institutions. To this purpose, two approaches were suggested. In the first one, systemic risk measures can be interpreted as the minimal amount of cash needed to secure a system \textit{after aggregating} individual risks. In the second approach, systemic risk measures can be interpreted as the minimal amount of cash that secures a system by allocating capital to each single institution \textit{before aggregating} individual risks. Although the theory behind these risk measures has been well investigated by several authors, the numerical part has been neglected so far.
In this paper, we use stochastic algorithms schemes in estimating Multivariate Shortfall Risk Measure (MSRM) and prove that the resulting estimators are consistent and asymptotically normal. We also test numerically the performance of these algorithms on several examples. 
  }

\bigskip


\section*{Introduction}
The axiomatic theory of risk measures, first initiated by the seminal paper of \cite{artzner}, has been widely studied over the past decades. The Value-at-Risk (VaR) is one of the best known and most common risk measures used by practitioners and regulation authorities. However, the  VaR lacks one important property: it does not take into account the diversification effect. To circumvent this problem, the VaR was replaced by the Conditional Value-at-Risk (CVaR) and a more general framework of improved risk measures has been introduced: \textit{Utility-based Shortfall Risk} (SR).
Nevertheless, when it comes to a system of financial institutions or portfolios, the question about how to assess the global risk, as well as individual risks, arises. Following the 2008 crisis, the traditional approach of measuring systemic risk that consists in considering each institution as a single entity isolated from other institutions, has shown some limits. Indeed, with this approach, the risk associated with a vector of positions $X=(X_1,...,X_d)$ can be written as:
$$ R(X) \vcentcolon= \sum_{i=1}^d \eta_i(X_i),$$
where each $\eta_i$ is a univariate risk measure. Then, \cite{chen2013axiomatic} proposed an approach whose spirit is  very close to the axiomatic framework initiated by \cite{artzner}. They showed that any systemic risk measure verifying certain axioms can be written as  the composition of a univariate risk measure $\eta$ and an aggregation $\Lambda$,  i.e.,
$$R(X) = \eta(\Lambda(X)).$$
The previous representation is known as the \q{Aggregate then Add Cash} approach as it consists first in aggregating the positions $X_1,...,X_d$ through the aggregation function $\Lambda$ and then  apply a univariate risk measure. One of the most common ways to aggregate the outcomes $X_i$ is to simply take the sum, that is to consider  $\Lambda(x) = \sum_{i=1}^d x_i$. It is worth noticing that, while summing up profit and losses might seem reasonable from the point of view of a portfolio manager because portfolios' profits and losses compensate each other, this aggregation rule seems inadequate from the point of view of a regulator where cross-subsidization between institutions is not realistic, since no institution will be willing to pay for the losses of another one.\\
Motivated by these considerations, \cite{unified} proposed another approach to measure the systemic risk. They first considered the systemic risk as the minimal capital that secures the system by injecting capital into the single institutions, \textit{before aggregating the individual risks}:
\begin{equation}
\label{eq0}
    R(X) \vcentcolon= \inf \{\sum_{i=1}^d m_i,~\Lambda(X + m) \in \mathbb{A} \},
\end{equation}
where $\mathbb{A}$ is an acceptance set. This approach, known as \q{Add Cash then Aggregate} consists in adding the amount $m_i$ to the financial position $X_i$ before the corresponding total loss $\Lambda(X + m)$ is computed. The systemic risk is then measured as the minimal total amount $\sum_{i=1}^d m_i$ injected into the institutions to make it acceptable. With this approach, a joint measure of total risk as well as individual risk contributions to systemic risk is obtained. If $m^* = (m_1^*,..., m_d^*)$ is an optimum, that is $R(X) = \sum_{i=1}^d m_i^*$ and $\Lambda(X +m^*) \in \mathbb{A}$, one could  rank the $m_i^*$'s and hence be able to say that institution/portfolio $i$ requires more cash allocation or is riskier than $j$ if $m_i^* \ge m_j^*$.\\
In this article, we are interested in the numerical computation of the multivariate shortfall risk measure (MSRM) that was introduced in \cite{armenti}. They are an extension of univariate SR measures, and can be obtained by taking the aggregation function $\Lambda(x) = l(-x)$ where $l$ is a multivariate loss function (see Section \ref{2}) and the acceptance set $\mathbb{A} = \{X \in L^0(\R), E[X] \le 0 \}$.\\
In \cite{armenti},  MSRMs and their optimal allocations are approximated using a two step procedure. The expectation in the aggregation function  is first approximated using  Monte Carlo, Fourrier or Chebychev schemes. The minimum in equation \eqref{eq0} is then found by  a  built-in minimizing algorithm in python. However, the minimizing algorithm can be sensitive to the starting point of the algorithm, and there are no theoretical results on the convergence of this procedure. In the univariate case, SR can also be characterized as the unique root of a function $g:\R \mapsto \R$ that is expressed as an expectation, which can be estimated by  using a deterministic root finding algorithm combined with a Monte Carlo procedure that estimates $g(s)$ for all  $s \in \R$. \\
Another idea to compute  the VaR and CVaR comes from the fact that they are solutions and the value of the same convex optimization problem as pointed out in \cite{rockafellar}. Moreover, as they can be expressed as an expectation, this led \cite{pages} to define consistent and asymptotically normal estimators of both quantities using a classical Robbins-Monro (RM) procedure.  For SR, a constrained RM algorithm has also been proposed by \cite{stochastic}.\\
In this paper, we introduce stochastic algorithms for MSRM, and prove their consistency and asymptotic normality. Indeed, the optimal allocations of MSRM can also be characterized as the root of a function that is expressed as an expectation. However, the multivariate setting adds additional constraints:
the optimal allocations are characterized as the solution of a first order condition of the Lagrangian associated with the MSRM. Hence, the new problem is to find a saddle point $z^*=(m^*,\lambda^*)\in \R^{d} \times \R^+$, which is the unique  root of a function $h(z) = E[H(X,z)]$ determined by the first order conditions.\\
RM stochastic algorithms have been the subject of a vast literature (see e.g.   \cite{benveniste}, \cite{duflo}). The basic paradigm in its simplest form is to approximate $z^*$ by the following stochastic difference equation: $Z_{n+1} = Z_n + \gamma_n Y_n$, where $\gamma_n > 0$ is the step size that goes to zero as $n \rightarrow \infty$, and $Y_n = H(X_{n+1}, Z_n)$, where $(X_n)$ is a sequence of i.i.d random variables with the same law as $X$. Moreover, if the random variable $X$ is not directly simulatable, \cite{frikha2016multi}  extended the scope of multi-level Monte Carlo to the framework of stochastic algorithms.  In many cases, the analysis of these algorithms is based on  the so-called Ordinary Differential Equation (ODE) method introduced by \cite{ljung}. The main idea is to show that, in the long run, the noise is eliminated so that, asymptotically, the behaviour of the algorithm is determined by that of the average ODE: $\dot{z} = h(z)$. \\
To ensure the convergence of RM algorithms, the  function $h$ is assumed to have sub-linear growth. This assumption is often not verified for MSRM, due to the choice of loss functions. In order to avoid reducing drastically the scope of the applications, we introduce a constrained version of the  RM algorithm for MSRM. 
This consists in using the projection into a compact $K$ each time the sequence $Z_n$ goes out of $K$. First introduced by \cite{sanvicente}, an excellent survey on projection techniques can be found in \cite{kushner}. 
In the case of MSRM, the convergence of the constrained algorithm toward the optimal allocation $z^*$ is then closely related to the study a projected multivariate ODE, whose solutions remain on $K$. We provide a rigorous analysis of the asymptotic stability of this ODE, based on the introduction of a Lyapunov function and on the extension of the Lasalle invariant  principle to this ODE with discontinuous r.h.s. Under additional assumptions, we also establish the asymptotic normality of the estimator. \\
The analysis and estimation of the asymptotic covariance matrix is also more difficult due to the multivariate setting. By introducing an averaged ``Polyak-Ruppert" version of the stochastic algorithm, we obtain a simpler expression of the asymptotic covariance matrix and propose an estimator allowing the computation of empirical confidence intervals.   \\
Stochastic algorithms also offer a flexible framework that can be conveniently combined with variance reduction techniques such as Importance Sampling (IS).  \cite{stochastic} have adapted the RM for unidimensional SR in order to include IS. However, in order to apply this non parametric approach, it is necessary to find beforehand an appropriate  change of measure for each parameter $z$, which can be very difficult even in the one dimensional case.  Inspired by the parametric approach of \cite{pages} for the VaR/CVaR, we also propose an adaptive RM algorithm based on recursive IS. This also  provides a general variance reduction method for univariate SR, which doesn't require prior knowledge of the change of measure.\\
The paper is organized as follows. Section \ref{2}, is dedicated to MSRM. In particular, we recall the main theorem characterizing the optimal allocations. The main results of consistency and asymptotic normality of the RM algorithm are presented in Section \ref{3},  along with the Polyak-Ruppert averaging and adaptative importance sampling. Finally, section \ref{4} is devoted to some numerical experiments of our procedures. We present a first testing example with an exponential loss function,  where we have a closed formula for optimal allocations. We also give a second example using a loss function with a mixture of positive part and quadratic functions.

\section{About Multivariate Risk Measures}
\label{2}
Let $(\Omega, \mathcal{F}, P)$ be a probability space, and denote by $L^0(\R^d)$ the space of $\mathcal{F}$-measurable $d$-variate random variables on this space with $d \ge 2$. For $x, y \in \R^d$, we say that that $x \ge y$ ( $x > y$ resp.) if $x_k \ge y_k$ ($x_k > y_k$ resp.) for every $1 \le k \le d$. We denote by $||\cdot||$ the Euclidean norm, and $x \cdot y = \sum x_k y_k$. For a function $f: \R^d \mapsto [-\infty, \infty]$, we denote by $f^*(y) = \sup_{x}\{x\cdot y - f(x)\}$ the convex conjugate of $f$.
The space $L^0(\R^d)$ inherits the lattice structure of $\R^d$ and therefore, we can use the classical notations in $\R^d$ in a $P$-almost-sure sens. We say, for example, for $X, Y \in L^0(\R^d)$, that $X\ge Y$ (or $X > Y$ resp.) if $P(X\ge Y)= 1$ (or $P(X > Y) = 1$ resp.). To simplify the notation, we will simply write $L^0$ instead of $L^0(\R^d)$. Now, let $X=(X_1,...,X_d)\in L^0$ be a random vector of financial losses, i.e., negative values of $X_k$ represents actually profits. We want to assess the systemic risk of the whole system and to determine a monetary risk measure, which will be denoted $R(X)$, as well as a risk allocation $RA_k(X), k=1,...,d$ among the $d$ risk components. Inspired by the univariate case introduced in \cite{follmer2002convex}, \cite{armenti} introduced a multivariate extension of shortfall risk measures by the means of loss functions and sets of acceptable monetary risk allocations.
\begin{Definition}
A function $l:\R^d \mapsto (-\infty, \infty]$ is called a loss function if:
\begin{enumerate}[start=1,label={(A\arabic*)}]
\item $l$ is increasing, that is $l(x) \ge l(y)$ if $x \ge y$;
\item $l$ is convex and lower-semicontinuous with $\inf l < 0$;
\item $l(x) \ge \sum x_k - c$ for some constant $c$.
\end{enumerate}
Furthermore, a loss function $l$ is said to be permutation invariant if $l(x) = l(\pi(x))$ for every permutation $\pi$ of its components.
\end{Definition}
\textbf{Comment}: The property (A1) expresses the normative fact about the risk, that is, the more losses we have, the riskier is our system. As for (A2), it expresses the desired property of diversification. Finally, (A3) says that the loss function put more weight on high losses than a risk neutral evaluation.
\begin{Example}
Let $h :\R \mapsto \R$ be one dimensional loss function satisfying condition (A1), (A2) and (A3). We could build a multivariate loss function using this one dimensional loss function in the following way:
\begin{enumerate}[start=1,label={(C\arabic*)}]
\item $l(x) = h(\sum x_k)$;
\item $l(x) = \sum h(x_k)$;
\item $l(x) = \alpha h(\sum x_k) + (1 - \alpha) \sum h(x_k)$ for $0 \le \alpha \le 1$.
\end{enumerate}
\end{Example}
More specifically, in (C1), we are aggregating losses before evaluating the risk, whereas in (C2), we evaluate individual risks before aggregating. The loss function in (C3) is a convex combination of those in (C1) and (C2).\\
One of the main examples we will be studying in this paper are the two following ones:
$$h(x) = \frac{1}{\alpha + 1}(\sum_i e^{\beta x_i} + \alpha e^{\beta \sum x_i}) - \frac{\alpha + d}{\alpha + 1},~~h(x)=\sum_i x_i + \frac{1}{2}\sum_i (x_i^+)^2 + \alpha \sum_{i < j} x_i^+x_j^+,$$
\noindent where the coefficient $\alpha > 0$ is called the systemic weight and $\beta > 0$ is a risk aversion coefficient.\\
In the following, we will consider multivariate risk measures defined on Orlicz spaces (see \cite{orlicz} for further details on the theory of Orlicz spaces). This has several advantages. From a mathematical point of view, it is a more general setting than $L^\infty$, and in the same time, it simplifies the analysis especially for utility maximization problems. Therefore, we will consider loss vectors in the following multivariate Orlicz heart:
$$M^\theta = \{ X\in L^0:~ E[\theta(\lambda X)] < \infty, ~ \forall	\lambda > 0\},$$
where $\theta(x) = l(|x|), x \in \R^d$. See Appendix in \cite{armenti} for more details about Orlicz spaces.\\
Next, we give the definition of multivariate shortfall risk measures as it was introduced in \cite{armenti}.
\begin{Definition}
Let $l$ be a multivariate loss function and $X\in M^\theta$, we define the acceptance set $\mathcal{A}(X)$ by:
$$\mathcal{A}(X) \vcentcolon= \{ m \in \R^d:~ E[l(X-m)] \le 0\}.$$
The multivariate shortfall risk of $X\in M^\theta$ is defined as:
\begin{equation}
\label{R}
R(X) \vcentcolon = \inf\left\{ \sum m_k:~m\in \mathcal{A}(X) \right\} = \inf\left\{ \sum m_k:~ E[l(X-m)] \le 0\right\}.
\end{equation}
\end{Definition}
\begin{Remark}
When $d=1$, the above definition corresponds exactly to the univariate shortfall risk measure in \cite{follmer}.
\end{Remark}
The following theorem from \cite{armenti} shows that the multivariate shortfall risk measure has the desired properties and admits a dual representation as in the case of univariate shortfall risk measure. We introduce $Q^{\theta^*}$ the set of measure densities in $L^{\theta^*}$, the dual space of $M^\theta$:
$$ Q^{\theta^*} \vcentcolon=\left\{ \frac{dQ}{d P}\vcentcolon= (Z_1,...,Z_d),~Z\in L^{\theta^*},~ Z_k\ge 0~\text{and}~E[Z_k]=1~\text{for every k}\right\}.$$
\begin{Theorem}{[Theorem 2.10 in \cite{armenti}]}
The function
$$R(X) \vcentcolon = \inf\left\{ \sum m_k:~m\in \mathcal{A}(X) \right\},$$
is real-valued, convex, monotone and translation invariant. Moreover, it admits the dual representation:
$$R(X) = \underset{Q \in Q^{\theta^*}}{\max} \{E_{Q}[X] - \alpha(Q)\}, ~~X \in M^\theta,$$
where the penalty function is given by
$$ \alpha(Q) = \underset{\lambda > 0}{\inf}~E\left[\lambda l^*\left(\frac{dQ}{\lambda d P}\right)\right],~~ Q\in Q^{\theta^*}.$$
\end{Theorem}
Now, we address the question of existence and uniqueness of a risk allocation which are not straightforward in the multivariate case. \cite{armenti} showed that if the loss function is permutation invariant, then risk allocations exist and they are characterized by Kuhn-Tucker conditions. We denote by $Z = \{m \in \R^d, \sum m_i = 0\}$ the \textit{zero-sum allocations} set.
\begin{Definition}
A risk allocation is an acceptable monetary risk allocation $m\in \mathcal{A}(X)$ such that $R(X) = \sum m_k$. When a risk allocation is uniquely determined, we denote it by $RA(X)$.
\end{Definition}
\noindent We make the following assumption on the loss function $l$ and the vector of losses $X \in M^\theta$:
\begin{enumerate}[label=($\mathcal{A}$\textsubscript{l})]
\item
\begin{enumerate}[label = \roman*.]
\item\label{MainAssumption} For every $m_0$, $m \mapsto l(X-m)$ is  differentiable at $m_0$ a.s.;
\item $l$ is permutation invariant.
\end{enumerate}
\label{Al}
\end{enumerate}
\begin{Theorem}{[Theorem 3.4 in \cite{armenti}]}
\label{Main}
Let $l$ be a loss function and $X \in M^{\theta}$ such that assumption \ref{Al} holds. Then, risk allocations $m^* \in \mathbb{R}^d$ exists and they are characterized by the first order conditions:
$$1 = \lambda^* \mathbb{E}[\nabla l(X - m^*)],~~~\mathbb{E}[l(X-m^*)] = 0,$$
where $\lambda^* \ge 0$ is a Lagrange multiplier.
If moreover $l(x + \cdot)$ is strictly convex along zero sum allocations for every $x$ such that $l(x)\ge 0$, the risk allocation is unique.
\end{Theorem}
\textbf{Comment}: Let $f_0(m) = \sum_{i=1}^d m_i$ and $f_1(m) \vcentcolon = E[l(X-m)]$, for $m \in \R^d$ and $X \in M^\theta$. The assumption \ref{Al}-\ref{MainAssumption} together with the convexity of the function $m \mapsto l(X-m)$, we have that, by Theorem 7.46 in \cite{shapiro}, $f_1$ is differentiable at every $m\in \R^d$ and that,
$$\nabla f_1(m) = -E[\nabla l(X-m)],~m \in \R^d.$$
Therefore, the first order conditions given in the above theorem are equivalent to :
$$\left\{
\begin{aligned}
&\nabla f_0(m^*) + \lambda^* \nabla f_1(m^*) = 0,\\
& f_1(m^*) = 0.
\end{aligned}\right.$$
Furthermore, we also know, thanks to Theorem 28.3 in \cite{rockafellarConvex}, that the above conditions are equivalent to saying that $(m^*, \lambda^*)$ is a saddle point of the Lagrangian associated to the problem in \eqref{R}, 
\begin{equation}\label{lagrangian}
L(m, \lambda) \vcentcolon = f_0(m) + \lambda f_1(m)=\sum_{i=1}^d m_i + \lambda E[l(X-m)].
\end{equation}
Under the assumptions of the above theorem  $z^* \vcentcolon= (m^*, \lambda^*)$ is the unique solution of $h(z) = 0$, where:
\begin{equation}
\label{defh}
h(z) \vcentcolon=\left( \begin{aligned}
\lambda E[&\nabla l(X-m)] - 1\\
&E[l(X-m)]
\end{aligned}\right), ~z = (m, \lambda) \in \R^d \times [0, \infty[. 
\end{equation}

Thus, in order to find the unique risk allocation $m^*$, we can look for the zeros of the function $h$. We introduce in the section a stochastic algorithm in order to approximate $z^*$, with the advantage of being incremental, less sensitive to dimension, and  with offer a flexible framework that can be conveniently combined with features such as importance sampling, as discussed in Section \ref{SectionIS}. 
\section{Multivariate Systemic Risk Measures and Stochastic Algorithms}
\label{3}
Let $l$ be a loss function   and a vector of losses $ X \in M^\theta$. In the remainder of the paper, we assume that the hypotheses  of Theorem \ref{Main} are verified by $l$, that is:
\begin{enumerate}[label=($\mathcal{A}^\prime$\textsubscript{l})]
\item
\begin{enumerate}[label = \roman*.]
\item For every $m_0$, $m \mapsto l(X-m)$ is  differentiable at $m_0$ a.s.;
\item $l$ is permutation invariant;
\item $m \mapsto E[l(X-m)]$ is strictly convex.
\end{enumerate} 
\label{Al'}
\end{enumerate}
Under \ref{Al'}, Theorem \ref{Main} ensures that there exists a unique risk allocation $m^*$ such that $z^*=(m^*, \lambda^*)$ is the unique root of the function $h(z) \vcentcolon= E[H(X, z)]$, where we set
\begin{equation}
\label{DefH}
H(X, z) = \begin{pmatrix}
\lambda \nabla_m l(X - m) - 1\\
l(X - m)
\end{pmatrix},~ X \in M^\theta.
\end{equation} 
Note that the assumptions ($\mathcal{A}^\prime$\textsubscript{l}) (i) and (iii) allow us to consider  loss functions of interest, which are not necessarily differentiable and/or strictly convex, such as the second example in  Section \ref{SecondExample}.

\paragraph{} The aim of this section is to construct a stochastic  algorithm that converges to the root $z^*=(m^*, \lambda^*)$ under some suitable assumptions. 
The standard Robbins-Monro (RM) procedure approximates the function  $h(z) = E[H(X, z)]$ at the $n+1$th step by  $Z_{n+1} = (m_{n+1}, \lambda_{n+1}) = Z_n + \gamma_n H(X_{n+1}, Z_n)$, with $X_{n+1}$ a random variable distributed as $X$. 
In order to obtain a.s. convergence of the algorithm toward the unique root $z^*$, one crucial condition  is the sublinearity of $h$, which is very constraining and does not allow us to define a  framework general enough to cover a wide range of loss functions. Consequently,  we propose a constrained version of the RM algorithm. 
The constrained algorithm requires the sequence $(Z_n)_{n \ge 1}$ to remain in a compact set $K$, for non-explosion reasons.  Each time $Z_n$ goes out of $K$, it  is  replaced by its projection on $K$ leading to the following projected RM algorithm:
\begin{equation}\label{RM2}
\begin{aligned}
Z_{n + 1} &= \Pi_K[Z_n + \gamma_n H(X_{n+1}, Z_n)],~ Z_0 = z_0 \in K\\
&= \Pi_K[Z_n + \gamma_n h(Z_n) +\gamma_n \delta M_n],
\end{aligned}
\end{equation}
where $\delta M_n = H(X_{n+1}, Z_n)-h(Z_n)$, and  $(X_n)_{n \ge 1}$ is an i.i.d sequence of random variables with the same distribution as $X$, independent of $Z_0$. An important point is that  the Algorithm \eqref{RM2} approximates simultaneously the optimal allocation $m ^*$ and the Kuhn-Tucker coefficient $\lambda^*$, which differs from the univariate case  of \cite{stochastic} where an algorithm which approximates $m^*$ directly can be implemented. We refer to Chapter 4 in \cite{kushner} for a general introduction on constrained stochastic algorithms.\\

 In the sequel, we denote $\mathcal{F}_n = \sigma(Z_0, X_i, i \le n)$. Note that $\delta M_n$ is a martingale difference sequence with respect to the filtration $\mathcal{F}=(\mathcal{F}_n)$. 
$(\gamma_n)_{n \ge 1}$ is assumed to be  a deterministic step sequence decreasing to zero and satisfying:
\begin{equation}\label{gamma}
\sum_{n \ge 1} \gamma_n = +\infty~\text{and}~\sum_{n \ge 1} \gamma_{n}^2 < + \infty.
\end{equation}
In the following, we will take  $\gamma_n = \frac{c}{n^\gamma}$  where $c$ is a positive constant and $\gamma \in ]\frac{1}{2}, 1]$.

Finally,  $K $ is an hyperrectangle such that $z^*$ is in the interior of $K$: $K=\{m \in \R^d, a_i \le m_i \le b_i\}\times [0, A]$. 
 This assumption can be difficult to verify in specific examples. However, the projection on the compact set $K$ can be replaced by an adaptive projection procedure, the so-called projection \q{à la Chen}, introduced in \cite{Chen1986STOCHASTICAP}, \cite{chen1987convergence}.  Given an increasing sequence of compact sets $(K_n)_n$ that covers $\R^d$,  each time  $Z_n$ leaves the current compact set $K_{\sigma(n)}$,  the compact set is replaced by $K_{\sigma(n) + 1}$ before the next step of the algorithm: 
\begin{equation*}\label{RM2}
\begin{aligned}
& Z_{n + 1}  = \Pi_{K_{\sigma(n)}}[Z_n + \gamma_n H(X_{n+1}, Z_n)],~ Z_0 = z_0 \in K_0, \\
&K_{\sigma(n+1)} = \begin{cases}
  K_{\sigma(n)}  \text{ if }  Z_n \in   K_{\sigma(n)},  \\
   K_{\sigma(n)+1}  \text{ if }  Z_n \notin   K_{\sigma(n)}.
    \end{cases}
\end{aligned}
\end{equation*}
 \cite{lelong} proved Central Limit Theorems for this type of projection, that are similar to the ones with the projection over one single compact set (see also \cite{andrieu2005stability}). For the sake of simplicity, we will use here the classical projection over a fixed compact set $K$.

\subsection{Main Results}

We start by stating the two main results regarding the consistency and asymptotic normality of  Algorithm \eqref{RM2}. The proofs of Theorem \ref{asTheo} and \ref{aNorm} can be found in Section \ref{SectionTh1} and \ref{SectionProofTh2}.    \\
 Let $\sigma^2(\cdot)$, $\Sigma(\cdot)$ and  $m^{2+p}(\cdot)$, for $p > 0$, be defined as follows:
\begin{equation*}
\left\{
\begin{aligned}
\sigma^2(z) &= E[||H(X,z)-h(z)||^2];\\
m^{2+p}(z) &= E[||H(X,z)-h(z)||^{2+p}];\\
\Sigma(z) &= E[(H(X,z)-h(z))(H(X,z)- h(z))^\intercal].
\end{aligned}\right.
\end{equation*}
Recall that Assumption \ref{Al'} ensures the existence of a unique risk allocation $z^* =(m ^*, \lambda^*)$. We introduce the  following classical assumptions for stochastic algorithms  in order to obtain the almost sure convergence of $(Z_n)_n$ to $z^*$:
\begin{center}
\begin{minipage}{7cm}
\begin{enumerate}[label=($\mathcal{A}$\textsubscript{a.s.})]
\item 
\begin{enumerate}[label=\roman*.]
\label{Aas}
\item\label{Aas1} $h$ is continuous on $K$;
\item\label{Aas2} $\underset{z\in K}{\sup}~\sigma^2(z) < \infty$.
\end{enumerate}
\end{enumerate}
\end{minipage}
\end{center}
\begin{Theorem}[Almost sure convergence]
\label{asTheo}
Assume that the sequence $(Z_n)$ is defined by the  algorithm \eqref{RM2} and that assumptions \ref{Al'} and \ref{Aas} hold. Then, $Z_n \rightarrow z^*$ $P$- almost surely as $n \rightarrow \infty$.
\end{Theorem}
We introduce the following assumptions for the central limit theorem:
\begin{center}
\begin{enumerate}[label=($\mathcal{A}$\textsubscript{a.n.})]
\item
\begin{enumerate}[label = \roman*.]
\label{Aan}
\item\label{Aan1} $m \mapsto E[\nabla l(X-m)]$ is continuously differentiable. Let $A \vcentcolon= Dh(z^*)$ (Jacobian matrix of $h$ at $z^*$);
\item\label{Aan2} $(H(X_{n+1}, Z_n) \mathbf{1}_{|Z_n - z^*| \le \rho})$ is uniformly integrable for small $\rho > 0$;
\item\label{Aan3} For some $p > 0$ and $\rho > 0, \underset{|z-z^*| \le \rho}{\sup} m^{2+p}(z) < \infty$;
\item\label{Aan4} $\Sigma(\cdot)$ is continuous at $z^*$. Let  $\Sigma^*\vcentcolon= \Sigma(z^*)$ .
\end{enumerate}
\end{enumerate}
\end{center}

\begin{Remark}
\begin{enumerate}
\item Assumption ($\mathcal{A}$\textsubscript{a.s.}) (i) is satisfied for continuously differentiable loss functions $l$,  by the Lebesgue dominating convergence theorem.  In the case of  univariate SR (see e.g. \cite{stochastic}), the first assumption is also trivial since $h$ only depends on the continuous loss function $l$ and not on $\nabla l$. 
\item In the particular case when  $X$ is bounded, the assumptions ($\mathcal{A}$\textsubscript{a.n.})(ii), (iii) and (iv) are straightforward. 
\end{enumerate}
\end{Remark}

\begin{Theorem}[Asymptotic Normality]
\label{aNorm}
Recall that  the step sequence $(\gamma_n)_{n\geq 0 }$ in \eqref{RM2}  is defined by $\gamma_n = \dfrac{c}{n^{\gamma}}$. Assume that $\gamma \in (\frac{1}{2}, 1)$ and that assumptions \ref{Al'}, \ref{Aas} and \ref{Aan} hold. Then,
$$\sqrt{n^\gamma}(Z_n - z^*) \to \mathcal{N}\left(0, c^2\int_0^\infty e^{cA t}~\Sigma^*~e^{cA^\intercal t}dt\right).$$
If furthermore, $cA + \frac{I}{2}$ is a Hurwitz matrix and $cI - P$ is positive definite with $P$ solution to the Lyapunov's equation: $A^\intercal P + PA = -I$, then for $\gamma =1$, 
$$\sqrt{n}(Z_n - z^*) \to \mathcal{N}\left(0, c^2\int_0^\infty e^{(cA+\frac{I}{2})t}~\Sigma^*~e^{(cA^\intercal+\frac{I}{2})t}dt\right).$$
\end{Theorem}
\begin{Remark}\
\begin{enumerate} 
\item Note that, for convex optimization problems, where the matrix $A$ is symmetric negative definite, the two additional conditions reduce to the classical condition $cA + \frac{I}{2}$ is negative definite. Indeed, in this case, the solution of the Lyapunov's equation $A^\intercal P + PA = -I$ is simply $P = -A^{-1}/2$ and the condition $cI-P$ is positive definite, becomes equivalent to $cA + \frac{I}{2}$ is negative definite.
\item From a formal point of view, the choice $\gamma = 1$ gives the best rate of convergence. The asymptotic variance in this case depends on the constant $c$. We need to choose it such that $cA + \frac{I}{2}$ is a Hurwitz matrix and $cI - P$ is positive definite. Setting $c$ too small may lead to no convergence at all, while setting it too large, may lead to slower convergence as the effects of large noises early in the procedure might be hard to overcome in a reasonable period of time. 
\item  The learning rate $c$ in $\gamma_n = \frac{c}{n}$ can be replaced by a conditioning nonsingular matrix $\Gamma$. In this case,  the results of Theorem \ref{aNorm} still hold, where  $\gamma_n$ is replaced by a matrix $\Gamma_n= \Gamma/n$, and  with the modified asymptotic covariance matrix:
$$ \int_0^\infty e^{(\Gamma A+\frac{I}{2})t}~\Gamma\Sigma^*\Gamma^\intercal~e^{( A^\intercal \Gamma^\intercal+\frac{I}{2})t}dt.$$
In particular, if  $\Gamma A $ is  a  negative definite matrix,   then,  by the previous remark, the hypothesis  of Theorem \ref{aNorm}  are verified as soon as $\Gamma A + \frac{I}{2}$s a Hurwitz matrix. Actually,  the optimal choice of  conditioning matrix $\Gamma$, which is also called the gain matrix, is the one that will minimize the trace of asymptotic covariance.  This is done by taking $\Gamma = -A^{-1}$ which yields the asymptomatic optimal covariance: $A^{-1}\Sigma^* (A^{-1})^\intercal$.

However, the optimal choice of $\Gamma$ depends on the function $h$ and the equilibrium point $z^*$ which are unknown to us. Adaptive procedures that choose the matrix $\Gamma$ dynamically by estimating $Dh(z^*)$ adaptively have been suggested in the literature (see for example \cite{ruppert}), but are generally not as efficient as the Polyak-Ruppert averaging estimators discussed in Section \ref{sectionPR}.
\end{enumerate}
\end{Remark}
\subsection{Proof of Theorem \ref{asTheo}}
\label{SectionTh1}
The proof of the almost sure convergence of $(Z_n)$ to $z^*$ rely on the so-called ODE method for the study of constrained RM algorithms.  Introduced by Ljung and Kushner, the main idea of the method consists in replacing the analysis of the sequence $(Z_n)$ by the analysis of continuous-time interpolations of the sequence of stochastic approximations. As detailed in  Theorem \ref{ThKusher1} of Appendix \ref{AppendixKushner}, these interpolations converge under suitable assumptions to a solution of the  projected ODE defined by : 
\begin{equation}\label{ODE}
\dot{z} = h(z) + C(z), ~C(z) \in -\mathcal{C}(z),
\end{equation}
where $\mathcal{C}(z)$ is the convex cone determined by the outer normals to the faces that need to be truncated at $z$, and $C(z)$ is the minimum force needed to bring back $z$ to $K$. In particular, $C(z(t))\neq  0$ only for $z(t) \in \partial K$. When $z(t) \in \partial K$ with $h(z(t))$ directed outside of $K$, then $C(z(t))$ take the ``smallest" value in $- \mathcal C(z)$ to maintain $z(\cdot)$ in $K$.
More precisely, for $z\in \R^{d+1}$ and $1\leq i \leq d$:
\begin{equation*}
C_i(z) = -h_i(z_i)\big( \mathbf{1}_{\{z_i = a_i\} \cup \{h_i(z) <0\}} + \mathbf{1}_{\{z_i = b_i\} \cup \{h_i(z) >0\}}\big). 
\end{equation*}
 Note that the existence of the projected ODE is obtained as a byproduct of the convergence of the SA algorithm (see Theorem \ref{ThKusher1} in Appendix).  
We refer to Section 4.3  in \cite{kushner}  for more details on the projected ODE, or to Section 2.2 of \cite{nagurney1995projected} for a more general presentation of projected ODEs.

Theorem 5.2.1 of \cite{kushner} (Theorem \ref{ThKusher1}) yields that for fixed $\omega \in \Omega$,  the constrained sequence $(Z_n(\omega))$ defined by \eqref{RM2} converges to a point in the limit set of the projected  ODE \eqref{ODE}. In order to prove Theorem \ref{asTheo}, the main point is thus to show that the limit set of projected ODE is reduced to $z^*$, or in other words, that $z^*$ is globally asymptotically stable in the sens of Lyapunov. This point, which is straigthforward in the univariate case, is more complicated for MSRM due to the introduction of the Lagrangian function $L$ from which $h$ derives. \\
Since $z^*$ is in the interior of $K$ and $h(z^*)=0$, $z^*$ is an equilibrium point for the projected ODE \eqref{ODE}. In order to study the asymptotic stability of the equilibrium $z^*$, we define in the following proposition some convenient Lyapunov function $V$. A natural and classical choice for this type of problems is $V(z) = ||z - z^*||^2$. Some recalls on Lyapunov stability and Lyapunov functions are presented in Section \ref{ODEMethod} (see also e.g. \cite{appnonlincont} for more details). The following proposition shows that its derivative along any state trajectory is negative semi-definite on $K$.
\begin{Proposition}\label{negDef}
The function $V(z) = ||z - z^*||^2$ is such that $z \to \dot{V}(z) = \langle \nabla V(z), h(z) + C(z) \rangle$ is negative semi-definite on $K$ with the respect to the ODE in \eqref{ODE}.\\
Furthermore,
\begin{equation}
\dot{V}(z) \leq - W(z):=  2\langle z-z^*, h(z) \rangle \leq 0 , \quad \forall z \in K.  
\end{equation} 
\end{Proposition}
\begin{proof} 
First, let $z = (m, \lambda) \in \interior(K)$ so that
 $\dot{V}(z) = \langle \nabla V(z), h(z)\rangle = 2 \langle z - z^*,  h(z)\rangle,$ and define $L$ the Lagrangian as defined in \eqref{lagrangian}. We have:
\begin{align*}
\langle  z - z^* , h(z)\rangle &=  \langle m - m^*, \lambda E[\nabla l (X -m)] - \mathbf{1}\rangle + (\lambda - \lambda^*) E[l(X-m)]\\
&= - \langle m - m^*, \partial_m L(m, \lambda)\rangle +(\lambda - \lambda^*)\partial_\lambda L(m,\lambda).
\end{align*}
Now, thanks to the convexity of $L$ with respect to $m$, we have:
$L(m^*, \lambda) \ge L(m, \lambda)+\langle m^* - m, \partial_m L(m,\lambda)\rangle$. This in turn implies that
$$ L(m^*,\lambda) - L(m, \lambda) + (\lambda - \lambda^*)\partial_\lambda L(m, \lambda) \ge \langle m^* - m, \partial_m L(m,\lambda)\rangle + (\lambda - \lambda^*)\partial_\lambda L(m, \lambda).$$
But, we also have,
\begin{align*}
L(m, \lambda)& = \sum m_i + \lambda E[l(X-m)]=\sum m_i + \lambda^* E[l(X-m)] +(\lambda - \lambda^*)E[l(X-m)]\\
& = L(m, \lambda^*) + (\lambda - \lambda^*)\partial_\lambda L(m, \lambda).
\end{align*}
The previous inequality becomes then 
$$\langle m - m^* , \partial_m L(m, \lambda)\rangle - (\lambda - \lambda^*)\partial_\lambda L(m, \lambda) \ge  L(m, \lambda^*) - L(m^*, \lambda)$$
The RHS of the last inequality is non-negative, because, $(m^*, \lambda^*)$ is a saddle point, that is $L(m^*,\lambda) \le L(m^*,\lambda^*) \le L(m, \lambda^*)$. Therefore, we get that,
\begin{equation}\label{neg}
\langle  z - z^* , h(z)\rangle \le 0.
\end{equation}
Moreover, since $L$ is strictly convex with respect to $m$, the inequality \eqref{neg} is strict if $m \neq m^*$.\\
 Finally,  let $z \in \partial K$. If $z_i = b_i$  (resp. $a_i$) and $h_i(z) > 0$ (resp. $h_i(z) < 0$) for some $i$, then $C_i(z) = -h_i(z)$. Since  $z^* \in  int(K)$, $a_i< z_i^* < b_i$ and thus  $(z_i - z^*_i)C_i(z)  \le 0$. 	Hence,  $\dot{V}(z) = 2\langle z-z^*, h(z) + C(z) \rangle \le \langle  z - z^* , h(z)\rangle \le 0$, which achieves the proof.
\end{proof}
\begin{Remark}
We cannot conclude that $\dot{V}$ is negative definite on $K$, because $z\neq z^*$ does not imply that $m \neq m^*$. Besides, if $z=(m^*, \lambda)$ such that $\lambda \neq \lambda^*$, we have $\dot{V}(z)=0$ and $z \neq z^*$. 
\end{Remark}
As stated in the previous remark, $\dot{V}$ is not negative definite, and thus the standard Liapunov theorem (see e.g. Theorem 3.3 in \cite{appnonlincont}) cannot be applied. In order to prove the asymptotic stability of $z^*$, we  apply  LaSalle's invariant principle. The invariant set theorem characterizes  the limit set of the ODE as the largest  invariant set in $\{z ; \; \dot{V}(z) = 0\}$. Note that,  due to projection on $K$, the r.h.s of the ODE \ref{ODE} is discontinuous.  Versions of Lasalle's  invariance principle for  very general nonsmooth ODEs can be found in \cite{fischer2013lasalle} (see also e.g.  \cite{filipov}, \cite{shevitz1994lyapunov}). In the case of \eqref{ODE},  the direct proof presented below is facilitated by the fact that $K$ is a compact set.
\begin{Proposition}
\label{asymptotStableProp}
The equilibrium point $z^*$ of the ODE \eqref{ODE} is globally asymptotically stable.
\end{Proposition}
\begin{proof} 
Let us first show that limit points of the ODE \eqref{ODE} are in the set  $ R = \{ z \in K ; \; W(z) = 0\}$. This can viewed as special case  of Corollary 1 in \cite{fischer2013lasalle}. Let $(z(t))$ be a solution to the projected ODE \eqref{ODE} with initial condition $z(0) \in K$. 

The first step of the proof is proving that $W(z(t))  \to 0$ when $t \to \infty$, with $W$ the scalar function defined by $W(z)= 2\langle z-z^*, h(z) \rangle$, by applying the Barbalat's Lemma \ref{BarbalatLemma}. 
By Proposition \ref{negDef}, $\dot V \leq 0$. Hence, $t\to V(z(t))$ is decreasing and 
\begin{equation}
\int_0^t W(z(s)) d s \leq  - \int_0^t \dot V(z(s)) d s =  V(z(0)) - V(z(t)). 
\end{equation} 
Since $V(z(t)) \geq 0$, the l.h.s of the previous inequality is increasing and bounded by $V(z(0))$, and thus $\lim_{t\to \infty}  \int_0^t W(z(s)) d s <\infty$. Since $z$ is absolutely continuous and $h$ is bounded on $K$, $z$ is uniformly continuous.  Finally, since $W$ is a continuous function, it is uniformly continuous on the compact $K$ and thus $t \to W(z(t))$ is uniformly continuous. By Barbalat's lemma, we can conclude that 
\begin{equation*}
\lim_{t\to \infty} W(z(t)) = 0. 
\end{equation*}
Thus, limit points of the ODE \eqref{ODE} are in the set  $R = \{ z \in K; \; W(z) = 0\}$. \\

Since solutions of \eqref{ODE}  are in the bounded domain $K$, the  limit set of the projected ODE is an invariant set (see e.g. Section 4.3 of \cite{kushner} or Lemma 12.5 in \cite{filipov}). It thus remains to show that $\{z^*\}$ is the only invariant set in $R$.
 Let $M$ be the largest invariant set in $R$, and $z = (m, \lambda) \in M$. As discussed in the proof of Proposition \ref{negDef}, if $(m, \lambda) \in R$, then necessarily $m = m^*$. Thus, since $M $ is an invariant set in $R$, if $z(0)=(m, \lambda) \in M$, then $z(t) =(m^*, \lambda(t)) \in R$ for all $t \ge 0$, with  $z(\cdot)$ solution of the following ODE,
\begin{equation}\label{sys}
\left\{
\begin{aligned}
& 0 = \lambda(t)E[\nabla l(X- m^*)] - \mathbf{1} + C(z(t))_m,~ t \ge 0,\\
& \frac{d\lambda(t)}{dt} = E[l(X-m^*)] + C(z(t))_\lambda,~ t \ge 0.
\end{aligned}\right.
\end{equation}
Now, since  $z^*=(m^*,\lambda^*) \in \interior(K)$, we get that, $C(z(t))_m= C((m^*, \lambda(t))_m = 0$. We obtain that $\lambda(t) E[\nabla_m l(X-m^*)] - \mathbf{1} = 0$, and consequently $t\to \lambda(t)$ is a constant function equal to $\lambda^*$,  since $\lambda^*$ is the only solution of the first order condition by Theorem \ref{Main}.

We have then showed that the limit set is simply $\{z^*\}$ and therefore $z^*$ is globally asymptotically stable equilibrium for the ODE \eqref{ODE}. 
\end{proof}
\begin{proof} [Proof of Theorem \ref{asTheo}]
By Proposition \ref{asymptotStableProp},  $z^*$ is the only point  in the limit set of ODE \eqref{ODE}. \\
Thus, if the assumptions \ref{AssumptionKushner1} (1)-(3) are verified,   we can conclude that $Z_n \rightarrow z^*$ as $n \rightarrow \infty$ by Theorem \ref{ThKusher1} in the Appendix.
 The first assumption \eqref{Kush1A1} of the theorem is exactly  the Assumption \ref{Aas}-(ii).  Assumption \eqref{Kush1A3} corresponds to the continuity of $h$, i.e. Assumption \ref{Aas}-(i). Finally, Assumptions \eqref{Kush1A4}  corresponds to the hypothesis \eqref{gamma} on the step sequence $(\gamma_n)$, which concludes the proof. 
\end{proof}
\subsection{Proof of the asymptotic normality}
The proof of Theorem \ref{aNorm} is obtained as an application of Theorem \ref{ThKushner2} in the Appendix (Theorem 10.2.1  in \cite{kushner}), with $g = ch$  and $Y_n = cH(X_{n+1}, Z_n)$ and $\epsilon_n = \frac{1}{n^{\gamma}}$. 
\label{SectionProofTh2}
\begin{proof} [Proof of Theorem \ref{aNorm}]
By the proof of Theorem \ref{asTheo}, Assumption \ref{AssumptionKushner1} is satisfied.  Let us prove that Assumption \ref{AssumptionKushner2} holds.  \\
First, let us start with the case $\gamma \in (\frac{1}{2}, 1)$. Then  Assumption \ref{AssumptionKushner2}.\ref{A1}  is true since $\mu_n = 1+ \frac{\gamma}{2n} + o(\frac{1}{n^2})$.  Assumption \ref{AssumptionKushner2}.2 corresponds to the  assumption \ref{Aan}-\ref{Aan2}. \\
Let  $(Z^n(\cdot))$ be the sequence of continuous time processes defined by \eqref{interpolation}. For all $t\geq 0$ and $n\in \N$,  $Z^n(t) \in \{Z_p ;  \;  p\geq n \}\subset K$. Hence, the sequence $(Z^n(\cdot))$ is uniformly bounded. Furthermore,  since $(Z_n)$ converges a.s. to $z^*$ by Theorem \ref{asTheo},  Aldous' criterion for tightness (see e.g. Theorem VI.4.5 in \cite{jacod2013limit}) is satisfied and hence the family  processes $(Z^n(\cdot))$ is tight. Finally, for all $t\geq 0$, $Z^n(t) = Z^0(t_n + t)$ converges a.s. to the constant $z^*$ by definition of the interpolation process $Z^0$. Thus,  $(Z^n(\cdot))$  converges  weakly to $z^*$, and Assumption \ref{AssumptionKushner2}.\ref{A3} holds.\\
Assumption \ref{AssumptionKushner2}.\ref{A4} corresponds to \ref{Aan}-\ref{Aan1}. The first  part of  {\color{blue} \ref{AssumptionKushner2}.\ref{A6}}  is \ref{Aan}-\ref{Aan3}. By  \ref{Aan}-\ref{Aan4} (continuity of $\Sigma$ at $z^*$) and since $Z_n \to z^*$ a.s., 
$$E[(Y_n-g(Z_n))(Y_n- g(Z_n))^\intercal\mathbf{1}_{|Z_n - z^*|\leq \rho}] 
 \to \Sigma_1 : = c^2 \Sigma^*. $$
Finally, \ref{AssumptionKushner2}.\ref{A9} holds for $\epsilon_n = \frac{1}{n^\gamma}$ (see e.g. p.341 in \cite{kushner}).\\
It remains to show that \ref{AssumptionKushner2}.\ref{A5} holds, that is the matrix $A' = cA$ is a Hurwitz matrix. We have:
$$A=\left (
   \begin{array}{c|c}
   	  \\
      \lambda^* D E[\nabla l(X-m^*)] & E[\nabla l(X-m^*)] \\
      \\
      \hline
      -E[\nabla l(X-m^*)] & 0
   \end{array}
\right) =  -\left( \begin{array}{c|c}
   	  \\
      \hat{A} & -\frac{1}{\lambda^*} \\
      \\
      \hline
      ~~\frac{1}{\lambda^*}~~& 0 
   \end{array}\right),$$
where $\hat{A}\vcentcolon=-\lambda^*DE[\nabla l(X-m^*)]$ corresponds to the second derivative of the Lagrangian $L$ with respect to $m$. Note that $L$ is strictly convex with respect to $m$ due to the strict convexity of $m \mapsto E[l(X-m)]$. This implies that $\hat{A}$ is a positive definite matrix. Thanks to Theorem 3.6 in \cite{saddle}, we deduce that $A$ (and thus $A'$) is a Hurwitz matrix.

For the case $\gamma = 1$, it remains to prove  Assumption \ref{AssumptionKushner2}.\ref{A9}, ie that $\lambda >1$ where $\lambda$  is the largest real number such that $I\geq \lambda P'$, with $P'$ the solution of the Lyapunov equation $(A')^\intercal P' + P'A' = -I$. Recalling $A' =cA$, $P'$ is  also solution of $A^\intercal cP' + cP'A = -I$.  By uniqueness of the Lyapunov equation's solution,  we obtain that $P=cP'$, with $P$ as defined in Theorem \ref{aNorm}.  Since $cI-P=c(I-P')$ is positive definite by assumption, $I-P'$ is positive definite and thus $\lambda>1$, which achieves the proof.
\end{proof}
\subsection{Polyak-Ruppert Averaging}
\label{sectionPR}
In order to ease the tuning of the step parameter which is known to monitor the numerical efficiency of RM algorithms, we are led to modify our algorithm and to use an averaging procedure. Averaging algorithms were introduced by Ruppert (see \cite{ruppert}) and Polyak (see \cite{polyak}) and then widely investigated by many authors. \cite{kushner} and \cite{yang} studied these algorithms in combination with projection and proved a Central Limit Theorem (CLT) for averaging constrained algorithms.\\
The following theorem describes the Polyak-Ruppert algorithm for MSRM and states its asymptotic normality. It is a direct consequence of theorem 11.1.1 in \cite{kushner}. 
\begin{Theorem}\label{aNormPoly}
Assume $\gamma \in (\frac{1}{2}, 1)$ and that assumptions \ref{Al'}, \ref{Aas} and \ref{Aan} hold. For any arbitrary $t > 0$, we define 
\begin{equation}\label{Poly}
\bar{Z}_n = \frac{\gamma_n}{t}\sum_{i=n}^{n + t/\gamma_n - 1} Z_i,
\end{equation}
where any upper summation index $u \in \R^+$ is interpreted as its integer part. If $\Sigma^*$ is positive definite, then we have the following CLT:
\begin{equation}\label{PolyNorm}
\sqrt{\frac{t}{\gamma_n}}\left(\bar{Z}_n - z^*\right) \underset{n\to \infty}{\longrightarrow} \mathcal{N}\left(0, V + O\left(\frac{1}{t}\right)\right),
\end{equation}
where $V = A^{-1}\Sigma^* (A^{-1})^\intercal$.
\end{Theorem}
\begin{Remark}\
\begin{enumerate}
\item In \eqref{Poly}, the \textit{window of averaging} is $t/\gamma_n$ for any arbitrary real $t > 0$. Equivalently, $\gamma_n \times$ (size of window) does not go to infinity as $n \to \infty$, hence the name \q{minimal window} of averaging. In contrast, the \q{maximal window} of averaging corresponds to a window size $q_n t$ such that $\gamma_n q_n \to \infty$. A natural and a classical choice is taking $\gamma_n = c / n^\gamma$ and $q_n =n$. In the case of \textit{maximal window} of averaging, under some extra conditions, we are able to achieve the optimal asymptotic variance without an extra term $O(1/t)$(see Theorem 11.3.1 in \cite{kushner}).
\item Two sided averages can also be used instead of the one-sided average in \eqref{Poly}.
\end{enumerate}
\end{Remark}
\paragraph{Estimator of asymptotic variance}
The previous CLT theorems assert that, under some suitable conditions, our RM and PR algorithms converge to the root $z^*$ with a corresponding rate. More specifically, in Theorem  \ref{aNormPoly}, the asymptotic variance $V$ depends on $\Sigma^*$ and $A$. In practice, these two quantities are unknown and need to be approximated in order to derive confidence intervals for our estimators.
In Theorem \ref{aNorm}, in both cases, $\gamma = 1$ and $\gamma \in (\frac{1}{2}, 1)$, the asymptotic variance is expressed as an infinite integral that involves $\Sigma^*$ and $A$. The numerical evaluation of these integrals is a non-trivial exercise even when $\Sigma^*$ and $A$ are known. In \cite{estimation}, they described an approach that produces confidence regions and that avoids the necessity of having to explicitly estimate these integrals.\\
In the following proposition, we provide consistent estimators of these two quantities. The proof relies mainly on the Martingale Convergence Theorem.
\begin{Proposition}\label{estVarJac}
Assume \ref{Al'}, \ref{Aas} and \ref{Aan} hold.\\
If $z \to E[||H(X,z)||^4]$ is locally bounded around $z^*$, then,
\begin{equation}\label{estVar}
\Sigma_n \vcentcolon= \frac{1}{n}\sum_{k=1}^{n}H(X_{k}, Z_{k - 1}) H(X_{k}, Z_{k - 1})^\intercal \to \Sigma^* ~a.s.
\end{equation}
Let $A_n^\epsilon$ the matrix whose elements $A_n^\epsilon(i, j)$ for $i, j \in \{1,..., d+1\}$ are defined as follows:
$$A_n^\epsilon(i,j) \vcentcolon = \frac{1}{\epsilon n}\sum_{k=1}^n H_i(X_k, Z_{k - 1} + \epsilon e_j) - H_i(X_k, Z_{k - 1} ),$$
then,
\begin{equation}\label{estJac}
\underset{\epsilon \to 0}{\lim}~\underset{n \to \infty}{\lim} A_n^\epsilon = A~a.s.
\end{equation}
\end{Proposition}
\begin{proof} 
Let $(S_n)_{n \in \N^*}$ be the sequence defined as:
$$S_n = H(X_n, Z_{n - 1}) H(X_n, Z_{n - 1})^\intercal - \Sigma(Z_{n-1})-h(Z_{n-1}) h(Z_{n-1})^\intercal, ~n \ge 1$$
 $(S_n)_{n \in \N^*}$ is a martingale difference sequence adapted to $\mathcal{F}$ and consequently the following sequence $(M_n)_{n \in \N^*}$ defined as:
$$M_n = \sum_{i=1}^n \frac{S_i}{i}, ~~n\ge 1,$$
is a $\mathcal{F}$-martingale. Moreover, since $Z_n \to z^*$ a.s., $h$ and $\Sigma$ are continuous in $z^*$ (Assumptions \ref{Aas}-\ref{Aas1} and \ref{Aan}-\ref{Aan4}),  and by the boundedness of $z \to E[||H(X,z)||^4]$ around $z^*$, we have that:
$$\underset{n \ge 1}{\sup}~E[||S_n||^2 | \mathcal{F}_{n-1}] < \infty~~a.s.$$
Thus, the martingale convergence theorem ensures the existence of a finite random variable $M_\infty$ such that $M_n \to M_\infty~a.s.$ Kronecker's lemma then guarantees that $\frac{1}{n}\sum_{i=1}^n S_i \to 0$. Now, since,
$$\Sigma_n = \frac{1}{n}\sum_{i=1}^n S_i +\frac{1}{n}\sum_{i=1}^n \Sigma(Z_{i-1}) +\frac{1}{n}\sum_{i=1}^n h(Z_{i-1}) h(Z_{i-1})^\intercal,$$
we deduce that $\Sigma_n \to \Sigma^*$.\\
The proof of \eqref{estJac} follows using the same arguments above.
\end{proof}
\begin{Remark}\
\begin{enumerate}
\item Instead of averaging on all observations, one could modify the estimators above and average only on recent ones. This might improve the behaviour of these estimators.
\item If we denote $V_n \vcentcolon= (A_n^\epsilon)^{-1} \Sigma_n ((A_n^{\epsilon})^{-1})^\intercal$, then we obtain an approximate confidence interval for PR estimator with a confidence of $1-\alpha$ in the following form:
\begin{equation}\label{confInt}
\left[\bar{Z}_{j, n} - \sqrt{\frac{V_{jj, n}}{t n^\gamma}}q_\alpha, \bar{Z}_{j, n} + \sqrt{\frac{V_{jj, n}}{t n^\gamma}}q_\alpha\right], j\in \{1...d\},~\gamma \in (0,1),
\end{equation}
where $q_\alpha$ is the $1-\frac{\alpha}{2}$ quantile of a standard normal random variable. Note that this confidence interval has the advantage of being obtained with one simulation run. For RM estimators, confidence intervals could be estimated empirically, however a naive approach would multiple runs of the SA algorithm.
\end{enumerate}
\end{Remark}

\subsection{Variance reduction with adaptive importance sampling (IS)}
\label{SectionIS}
The convergence of the stochastic algorithms presented in this section can be improved by implementing variance reduction techniques such as important sampling (IS).  The original idea of IS for the Monte-Carlo computation of an expectation $E[G(X)]$, is to find a change of probability measure $P'$, such that $E[G(X)]= E'[G'(X)] $ and $\textrm{Var}_{P'} (G'(X)) < \textrm{Var} (G(X))$. The main issue is then to find the right change of measure in order to reduce the variance. A possible parametric approach is to look for the optimal parameter $\theta^*$ minimizing the variance, among a family of distribution $(P_\theta)$.  Under some assumptions, the minimization problem  can be reinterpreted as finding the root of a function written as an expectation. This  is the main idea  of recursive IS by stochastic algorithm, where the parameter $\theta^*$ is estimated by a RM algorithm. \\
The Esscher transform (or exponential tilting) is  a natural tool for parametric IS, used extensively in actuarial science and for variance reduction. This transfom is a natural change of measure since it corresponds to a translation in the gaussian case, and a change of intensity for Poisson processes.  Let us first recall the main idea behind the Esscher transform and recursive IS in the standard case. For more details, see e.g. see e.g. \cite{Arouna}, \cite{Kawai2009}, or  \cite{LemairePages}. 
\paragraph{Esscher transform and parametric IS}  In the following, we assume that the $\R^d$-valued random variable $X$ has a density $f$.  For $\theta \in \R^d$, let $\psi$ be the log-Laplace transform of $X$, assumed to be finite: 
$$ \psi(\theta) = \log \big( E[\exp(\langle\theta, X\rangle]\big) < \infty. $$
Then, $\psi$ is convex  and infinitely differentiable. The Esscher transform is defined by the following change of measure introducing the density $g_\theta$: 
\begin{equation}
g_{\theta}(x) =\exp\left( \langle \theta, x\rangle - \psi (\theta) \right) f(x). 
\end{equation}
In the following we denote by $X^{(\theta)}$ the r.v. of density $g_\theta$. 
The expectation  can be replaced by 
$$E[G(X)] =E[G(X^{(\theta)})\exp \left(  - \langle X^{(\theta)}, \theta \rangle  + \psi(\theta) \right) ].$$
The main idea of the recursive IS procedure is to choose $\theta$ in order to  minimize the variance of  $G(X^{(\theta)})\exp \left(  - \langle X^{(\theta)}, \theta \rangle  + \psi(\theta) \right)$,  which boils down to minimizing the function 
\begin{align*}
 E[G^2(X^{(\theta)})\exp \left(  - 2\langle X^{(\theta)}, \theta \rangle  + 2\psi(\theta) \right)] = E[G^2 (X) \exp \left(  - \langle X, \theta \rangle  + \psi(\theta) \right) ]. 
\end{align*}
In other words, the optimal parameter is 
\begin{equation}
\theta^* = \text{argmin } E[G^2 (X,z) \exp \left(  - \langle X, \theta \rangle  + \psi(\theta) \right) ].
\end{equation}
Under some regularity assumptions, the minimization problem can be solved using stochastic algorithms. 
\paragraph{Adaptative RM algorithm with IS} Our setting of multivariate systemic risk measures is more complex. For $1 \leq i \leq d+1$, let $H_i$ be the $i$ component of the function $H$ defined by \eqref{DefH}. Then, for $z = (m, \lambda) \in \R^d \times [0,\infty[$ and $\theta \in \R^d$, 
\begin{equation}
\label{EqISEsscher}
h_i(z) = E[H_i(X,z) ] = E[H_i ( X^{(\theta)}, z) \exp \left(  - \langle X^{(\theta)}, \theta \rangle  + \psi(\theta) \right) ],
\end{equation}
Let  $\mathcal V = (\mathcal  V_i)_{1 \leq i \leq d}$, with $\mathcal  V_i$  the function defined for $z \in \R^d \times [0,\infty[$ and $\theta \in \R^d$ by
\begin{align}
\nonumber\mathcal   V_i(\theta,z)  & = E[H_i^2 ( X^{(\theta)}, z) \exp \left(  - 2\langle X^{(\theta)}, \theta \rangle  + 2\psi(\theta) \right) ] \\
 & = E[H_i^2 (X,z) \exp \left(  - \langle X, \theta \rangle  + \psi(\theta) \right) ]. 
\end{align}
In order to reduce the asymptotic variance obtained in Theorem \ref{aNormPoly}, it is natural to take $\theta_i= \theta_{i}^*$ in \eqref{EqISEsscher} such that 
\begin{equation}
\label{DefThetaOpt}
\theta_{i}^* = \text{argmin } \mathcal  V_i(\theta,z^*) = \text{argmin } E[H_i^2 (X,z^*) \exp \left(  - \langle X, \theta \rangle  + \psi(\theta) \right) ]. 
\end{equation}
Assume that for each $1\leq i \leq d$, 
\begin{equation*}
E[|X|H_i^2(X,z)\exp \left(  - \langle X, \theta \rangle  + \psi(\theta) \right)] <\infty, \quad \forall z\in \R^d\times [0,\infty[ \; ,\theta \in \R^{d}.
\end{equation*}
Then,  $\mathcal V(\cdot, z)$ is differentiable and   for $1\leq i \leq d$, 
\begin{equation}
\nabla \mathcal  V_i (\theta,z ^*) = E[(\nabla \psi(\theta) - X) H_i^2(X,z^*) \exp \left(-\langle \theta, X \rangle + \psi(\theta) \right)]. 
\end{equation}
\cite{LemairePages} (Proposition 3, see also \cite{pages}) give sufficient conditions for  $\mathcal  V_i(\cdot  ,z^*)$ to be convex and for the set $\{ \theta \in \R^d ; \; \nabla \mathcal V_i(\theta,z^*)=0\}$   to be non empty. 
We also refer to \cite{Kawai2009} in the special case when $X =\xi_T$ with $\xi$ a Levy process.  A RM algorithm as in \eqref{RM2} can be implemented to find the root of each gradient $\nabla \mathcal  V_i$, replacing $H_i$ with 
\begin{equation}
\label{EqFISVar}
F_i(X, z^*,\theta) = (\nabla \psi(\theta) - X) H_i^2(X,z^*) \exp \left(-\langle \theta, X \rangle + \psi(\theta) \right). 
\end{equation} 
However such an algorithm cannot be implemented in practice, since $z^*$ is unknown. The same issue appears in the case of the VaR/CVaR estimation in \cite{pages}. Hence, the authors  introduced an adaptive procedure combining the RM algorithm and recursive Importance Sampling,  by substituting the unknown parameters at the $n$th step the by their estimation at the previous step. This idea can be adapted to our framework of multivariate systemic risk measures. \\

Let us first introduce some notations. For $\Theta = (\theta_1, \dots , \theta_d) \in \R^{d(d+1)}$, let $\tilde{H}$ be the function replacing $H$ in 
\eqref{EqISEsscher}, defined by 
\begin{equation*}
\tilde{H}(X,z,\Theta) = \left(H_i(X,z)\exp\left( - \langle X, \theta_i \rangle + \psi(\theta_i)\right)\right)_{\{1\leq i \leq d+1\}}, 
\end{equation*}
and 
\begin{equation*}
F(X,z, \Theta) = \left( F_i(X,z,\theta_i) \right)_{\{1\leq i \leq d+1\}}, 
\end{equation*}
where $F_i$ has been defined in \eqref{EqFISVar}. Finally, we denote by $X^{(\Theta)}= (X^{(\theta_1)}, \dots, X^{(\theta_d)})$ the $\R^{d^2}$-valued random vector, where each $\R^d$ component $X^{(\theta_i)}$ has the density $g_{\theta_i}$ and is independent of the others. Finally,  let $\Sigma(z,\Theta)$ be the covariance matrix of $\tilde{H}(X,z,\Theta)$.

The new stochastic algorithm is defined as follow: 
\begin{align}
& \nonumber Z_{n+1} =\Pi_K [Z_n + \gamma_n \tilde H( X_{n+1}^{(\Theta_n)}, Z_n, \Theta_n) ], \\
& \label{ISalgo} \Theta_{n+1} =\Pi_{K'} [\Theta_n  + \gamma_n F(X_{n+1}, Z_n, \Theta_n) ],
\end{align}
where $K'$ is an hyper rectangle containing $0$ and $\Theta^*$. In the specific case of the VaR-CVaR approximation, \cite{pages} propose an unconstrained version of the algorithm, based on \cite{LemairePages}. However, the procedure cannot be applied here since $H$ does not verify the sublinearity condition, as discussed in Section \ref{3}.  However, as for $K$,  $K'$ can also be replaced by an increasing sequence of compact $(K'_{\sigma(n)})_{n \in \mathbb{N}}$. 

The proof of Theorem \ref{asTheo} and \ref{aNorm} can be adapted to obtain the a.s. convergence and asymptotic normality of the approximations $(Z_n,\Theta_n)_n$ introduced in \eqref{ISalgo}, as well as its Polyak-Ruppert counterpart. In particular,  the covariance matrix $\Sigma^* = \Sigma(z^*, 0)$ will be  replaced  by  $\Sigma(z^*, \Theta^*)$, minimizing the variance of the vector's $\tilde{H}(X,z^*,\cdot)$ components. \\
The asymptotic variance for the Polyak-Ruppert estimator with adaptative IS is actually $V(\Theta^*) = A^{-1}\Sigma(z^*, \Theta^*) (A^{-1})^\intercal$, with $A$ the Jacobian matrix of $h$ at $z^*$.   The matrix $A$ does not depend on the $\Theta^*$, but the diagonal elements of the covariance matrix $V$ might not minimize the function $V(\Theta) = A^{-1}\Sigma(z^*, \Theta^*) (A^{-1})^\intercal$.  
In the case of the joint computation of the VaR/CVaR, the expression of the matrix $A^{-1}$ is explicit, and this might  be the case for MSRM for specific loss functions. Another potential drawback of the adaptative RM with IS \eqref{ISalgo} is the number $d(d+1)$ of additional parameters introduced by the variance reduction procedure. Once again, the number of parameters could be  significantly reduced depending on the loss function $l$.  
\begin{Remark}
In dimension 1, it is straightforward to show that $\theta^*$ minimizes $V(\theta)$. Hence the adaptative algorithm \eqref{ISalgo}   also provides a general variance reduction method for univariate shortfall risk measures. This is very useful when the optimal change of measure is not known beforehand, which is the case most of the time.
\end{Remark}
\section{Numerical examples}
\label{4}
In this section, we test the performance of the proposed stochastic algorithms schemes for MSRM. 
We provide numerical examples showing the consistency properties of the different estimators and then their normal asymptotic behaviour,  with (PR algorithm) and without averaging (RM algorithm). Two examples are considered. In the first one, we consider a loss function of an exponential type coupled with a normal distribution. This example is relevant for our numerical analysis as we can explicitly express the optimal allocations in a closed form. In the second example, we consider a loss function that involves positive part function with a Gaussian and a compound Poisson distributions.  All the algorithms have been implemented on a standard computer using the programming language Python, and are publicly available at \url{https://github.com/AchrafTamtalini/SRM/}.\\
In the following, $n$ will denote the number of steps in one simulation run and $N$ the number of simulations. We introduce the following sequences:
\begin{equation}\label{average}
~\bar{D}_{n} \vcentcolon= \sqrt{t n^\gamma}(\bar{Z}_{n} - z^*)~~\gamma \in (\frac{1}{2}, 1),
\end{equation}
\begin{equation}
D_{n} \vcentcolon= \sqrt{n^\gamma}(Z_{n}-z^*),~\gamma \in (\frac{1}{2}, 1].
\end{equation}
\subsection{Toy example}
As a first simple example, we will consider  a exponential loss function of the following form:
\begin{equation}
\label{lossfunction1}
l(x_1,...,x_d) = \frac{1}{1+\alpha}\left[\sum_{i=1}^d e^{\beta x_i} + \alpha e^{\beta\sum_{i=1}^d x_i}\right] - \frac{\alpha + d}{\alpha + 1}
\end{equation}
We will set $d=2$ and consider a bivariate normal vector $X=(X_1, X_2)\sim \mathcal{N}(0, M)$ with $M = \begin{pmatrix}
\sigma_1^2 & \rho \sigma_1 \sigma_2\\
\rho \sigma_1 \sigma_2  & \sigma_2^2
\end{pmatrix}$. $\alpha$ is a systemic weight parameter taken to be non negative and $\beta > 0$ is the risk aversion coefficient. In this case, we can explicitly solve the first order conditions and derive closed formulas for optimal allocations. This will be useful to test our algorithms. The following Lemma extends Example 3.12 in \cite{armenti}:
\begin{Lemma}
Let $l$ be the loss function defined by \eqref{lossfunction1}, and $X$ the vector of random losses as above. Then, the optimal risk allocation is: 
\begin{equation}
m_i^* =
\left\{
\begin{aligned}
&\frac{\beta\sigma_i^2}{2},~~\text{if}~\alpha = 0,\\
&\frac{\beta\sigma_i^2}{2}+ \frac{1}{\beta}SRC(\rho, \sigma_1, \sigma_2, \alpha, \beta), ~~\text{if}~ \alpha > 0.
\end{aligned}\right. 
\end{equation}
With 
$$SRC(\rho, \sigma_1, \sigma_2, \alpha, \beta) = \ln\left(\frac{\alpha e^{\rho \beta^2\sigma_1 \sigma_2}}{-1 + \sqrt{1+\alpha(\alpha+2)e^{\rho\beta^2 \sigma_1 \sigma_2}}}\right).$$
\end{Lemma}
This shows that, in the case $\alpha > 0$, the risk allocations are disentangled into two components: an individual contribution $\frac{\beta\sigma_i^2}{2}$ and the \textit{Systemic Risk Contribution} term  (SRC). Note that taking $\alpha \to 0$ makes the SRC null since the systemic weight $\alpha$ is responsible for the systemic contribution in the loss function $l$. One can also show, by easy calculations, that the SRC is increasing with respect to $\rho$: the higher the correlation is,  the more costly the acceptable monetary allocations are. This could be explained by the fact that, with a higher correlation between the two components, the losses of one will induce the loss of the other and consequently the system will become riskier. Note also that we could also express in a closed form the Jacobian matrix $A$ and $\Sigma^*$.
\begin{proof}
The optimal risk allocations $m_i^*$ are characterized by the first order conditions given in Theorem \ref{Main}, i.e.
\begin{equation*}
\left\{
\begin{aligned}
&\frac{\lambda^*}{1+\alpha}E\left[e^{\beta(X_i - m_i^*)}+\alpha e^{\beta(X_1+X_2 - m_1^*-m_2^*)}\right]=1,~~ i=1, 2,\\
&\frac{1}{1+\alpha}E\left[e^{\beta(X_1-m_1^*)}+e^{\beta(X_2-m_2^*)} +\alpha e^{\beta(X_1+X_2-m_1^*-m_2^*)}\right]-\frac{\alpha+2}{\alpha+1}=0.
\end{aligned}\right.
\end{equation*}
The two first equations imply that $E[e^{\beta(X_1 - m_1^*)}] = E[e^{\beta(X_2 - m_2^*)}]$, which in turn gives that,
\begin{equation}\label{equality}
\frac{\beta \sigma_1^2}{2} - m_1^* = \frac{\beta \sigma_2^2}{2} - m_2^*.
\end{equation}
The third equation gives $e^{\frac{\beta^2\sigma_1^2}{2} - \beta m_1^*} +e^{\frac{\beta^2\sigma_2^2}{2} - \beta m_2^*} + \alpha e^{\rho\beta^2\sigma_1\sigma_2+\frac{\beta^2}{2}(\sigma_1^2+\sigma_2^2)-\beta(m_1^*+m_2^*)}=2+\alpha$. Thanks to \eqref{equality} and denoting $Q = e^{\frac{\beta^2 \sigma_i^2}{2} - \beta m_i^*}$, we get that, $\alpha e^{\rho\beta^2\sigma_1\sigma_2}Q^2 + 2Q - (2+\alpha)=0$. Taking the positive solution of the last equation gives $Q = \frac{-1 +\sqrt{1 +\alpha(\alpha+2)e^{\rho\beta^2\sigma_1\sigma_2}}}{\alpha e^{\rho\beta^2\sigma_1\sigma_2}}$. Now, denoting by $\mathrm{SRC} = -\log(Q)$, we obtain that $m_i^* = \frac{\beta \sigma_i^2}{2} + \frac{1}{\beta}\mathrm{SRC}$.
\end{proof}
In all this example, we fix $\alpha = 1$, $\beta =1$  and $\sigma_1 = \sigma_2 = 1$. With $\rho \in \{-0.5, 0, 0.5\}$, we obtain the exact values in the table below. Note that since we have $X_1\sim X_2 \sim \mathcal{N}(0, 1)$ and $l$ is permutation invariant, it follows that $m_1^*=m_2^*$. 
\begin{table}[H]
\caption{Exact optimal risk allocations.}
\centering
\begin{tabular}{|c|c|}
\hline
$\rho$ & $m_1^*=m_2^*$\\
\hline
$-0.5$ & $0.3868$\\
\hline
$0$ & $0.5$\\
\hline
$0.5$ & $0.6364$\\
\hline
\end{tabular}
\end{table}
For RM/PR algorithms, we used a number of steps $n = 10^5$. As for the compact $K$, we took $K=[0, 2]^3$  and $Z_0$ was taken uniformly on $K$. We run the different algorithms for $\gamma = 1$ and $\gamma = 0.7$. We chose an averaging parameter $t=10$ and we set $c=2$ in a first step. Figure \ref{AsFig1} shows that, for different values of $\rho \in \{-0.5, 0, 0.5\}$, our RM algorithm with $\gamma = 1$ converges relatively quickly to the optimal allocations, whereas when $\gamma = 0.7$, noise is still persisting. This is due to the step parameter $c$ as discussed in the previous section. In order to get a smoother numerical behaviour, two solutions are available to us:
either we use PR averaging (c.f. Figure \ref{AsFig1}), or we reduce the value of the parameter $c$ (Figure \ref{AsSmaller}). 
\begin{figure}[H]
    \centering
    \begin{subfigure}{0.4\textwidth}
        \includegraphics[width=\textwidth]{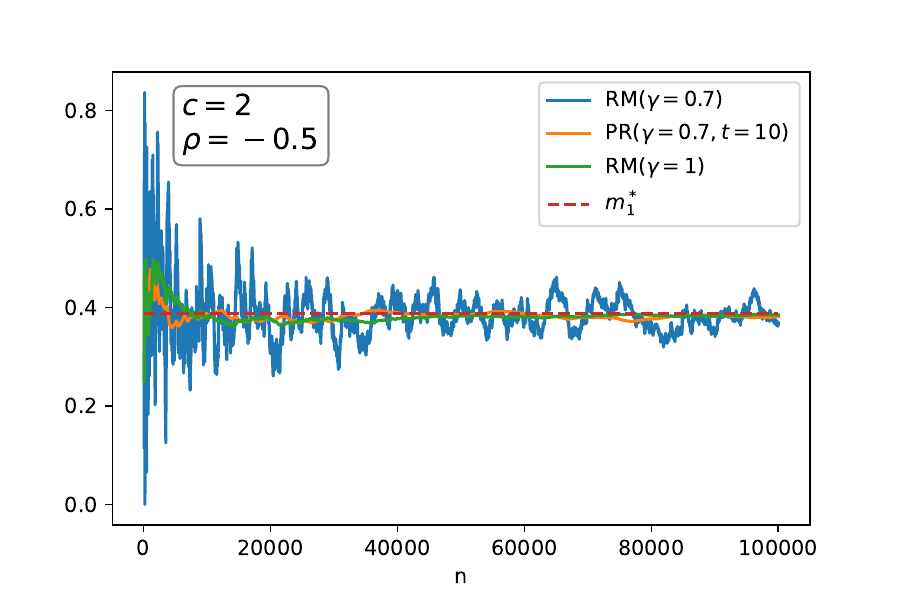}
    \end{subfigure}
    \begin{subfigure}{0.4\textwidth}
        \includegraphics[width=\textwidth]{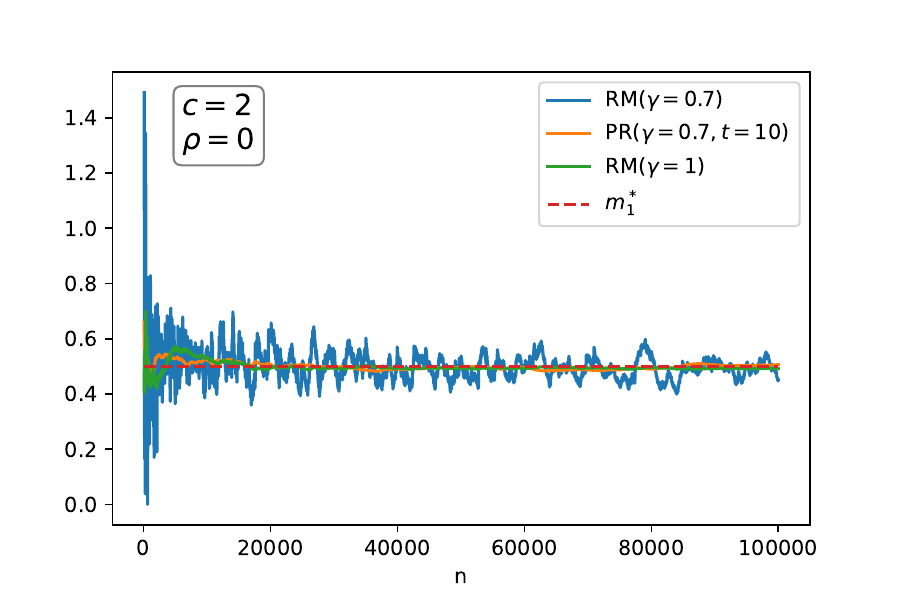}
    \end{subfigure}
    \begin{subfigure}{0.4\textwidth}
        \includegraphics[width=\textwidth]{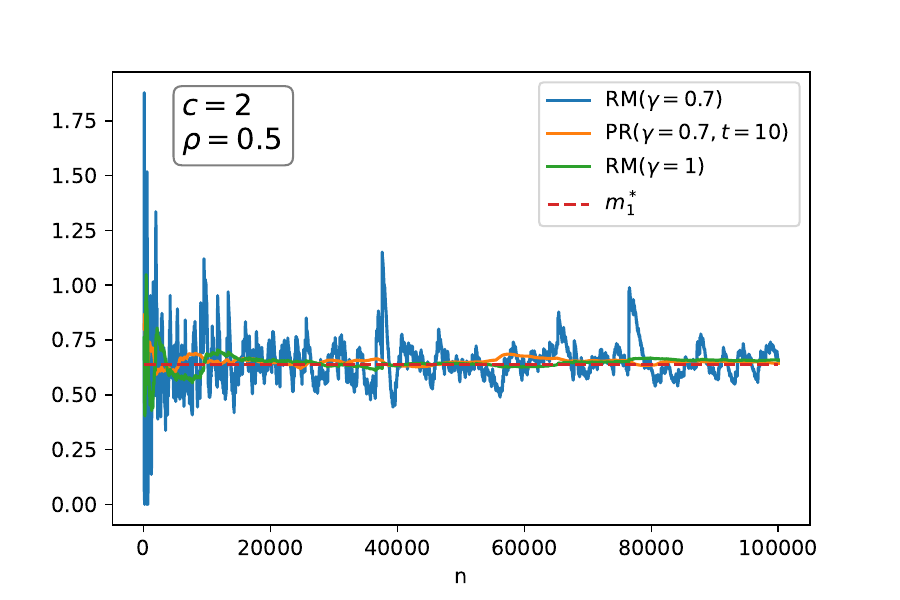}
    \end{subfigure}
    \caption{Consistency of RM/PR estimators with for different values of $\rho$.}
    \label{AsFig1}
\end{figure}
\begin{figure}[H]
    \centering
    \begin{subfigure}[h]{0.4\textwidth}
        \includegraphics[width=\textwidth]{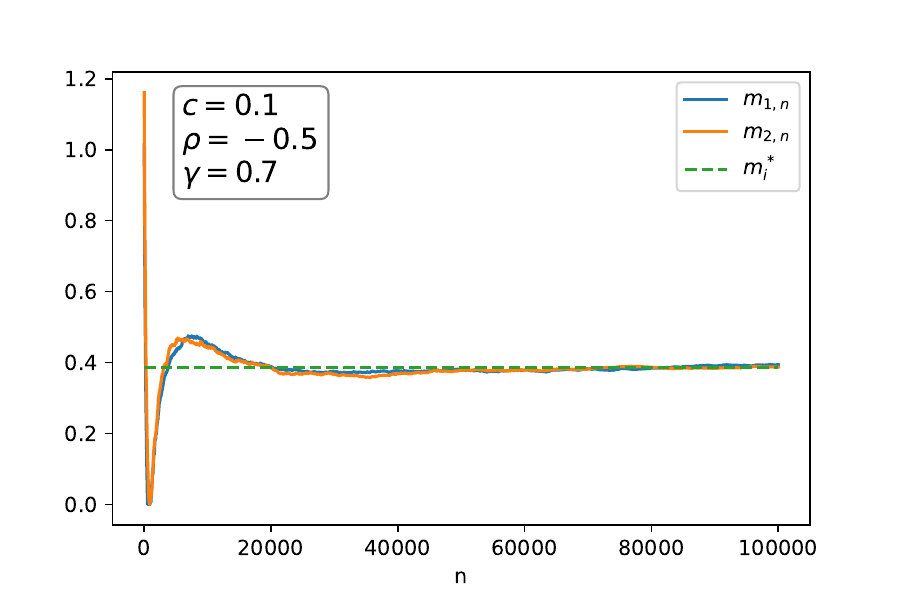}
    \end{subfigure}
    \begin{subfigure}[h]{0.4\textwidth}
        \includegraphics[width=\textwidth]{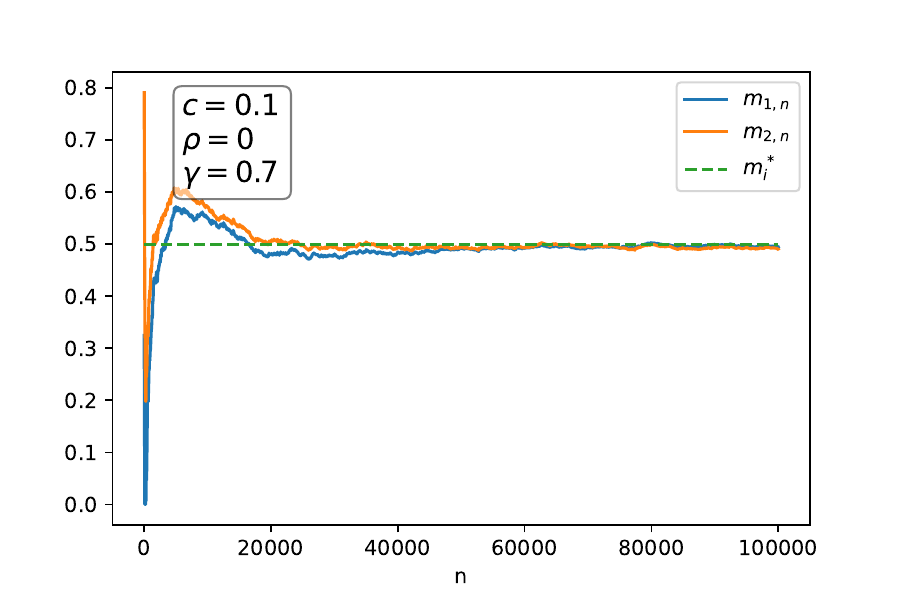}
    \end{subfigure}
    \begin{subfigure}[h]{0.4\textwidth}
        \includegraphics[width=\textwidth]{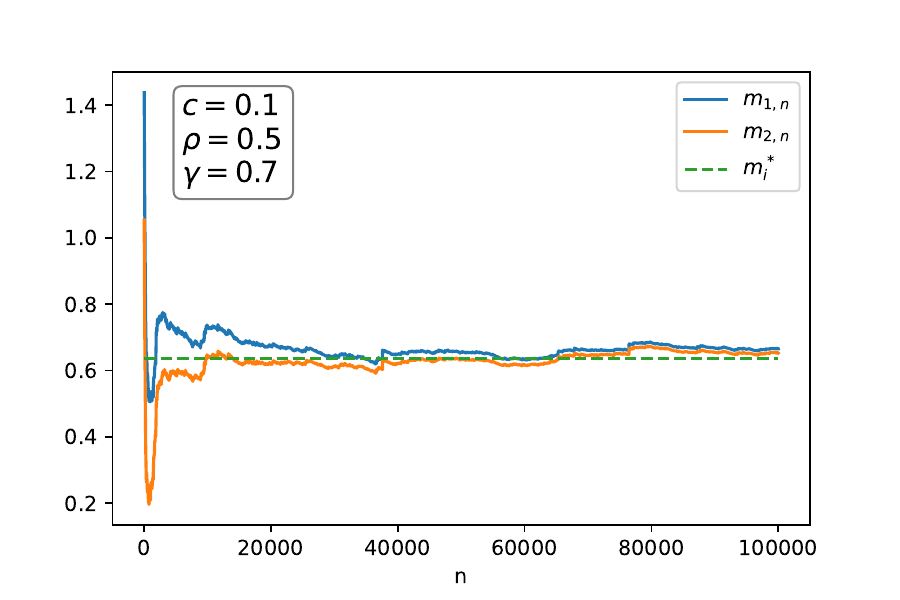}
    \end{subfigure}
    \caption{Consistency of RM estimators with $c = 0.1$ for different values of $\rho$.}
    \label{AsSmaller}
\end{figure}
Figure \ref{figCT} shows the increase of computational time when the error become smaller.  We have also compared in Table \ref{comparison} the computational time (CT)  of the RM algorithms (with $c=2$ and $\gamma=1$) with the Monte Carlo method used in \cite{armenti}. As it is often the case, the RM algorithm is slower than the Monte Carlo method in \cite{armenti}. However, as stated above, the latter does not provide any convergence results, and might not converge in practical applications.
\begin{figure}[H]
\centering
\includegraphics[width=0.5\textwidth]{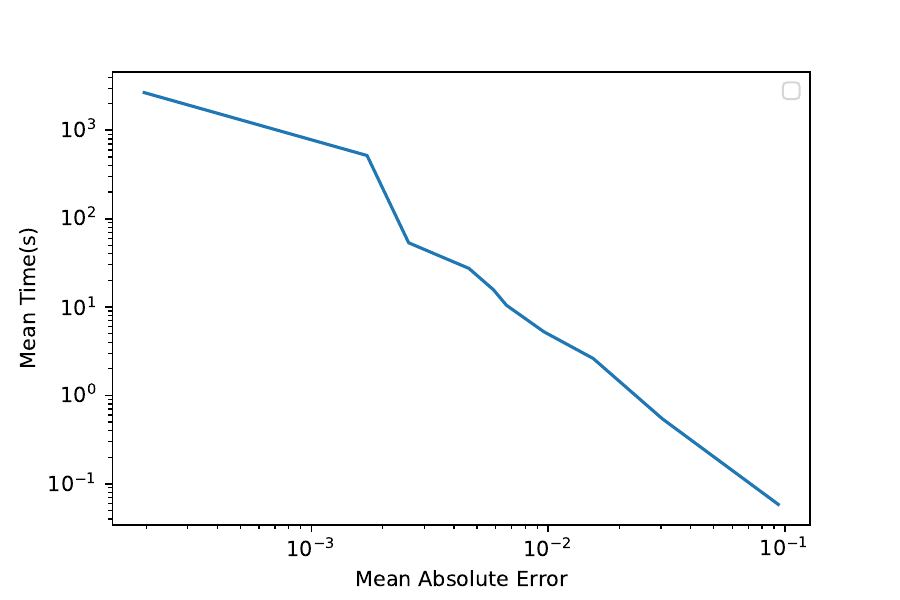}
\caption{Computational time  (in seconds) as a function of the (average) absolute error (log log scale), for the RM algorithm ($\gamma=1$, $c=2$ and $\rho=0.5$).}
\label{figCT}
\end{figure}
\begingroup
\setlength{\tabcolsep}{6pt} 
\renewcommand{\arraystretch}{1.2}
\begin{table}[H]
\centering
\begin{tabular}{|c|c|c|}
\hline
Absolute error & {CT RM (s)}  & { CT Monte Carlo (s)}\\
\hline
$2 \times 10^{-3}$ & 50 & 0.4 \\
\hline 
$2\times 10^{-4}$ & 2673 & 54 \\
\hline
\end{tabular}
\caption{ Comparison of computation times for RM algorithms ($\gamma=1$ and $c=2$) and Monte Carlo Method in \cite{armenti} ($m^*_1=m_2^*= 0.6364$).}
\label{comparison}
\end{table}
\endgroup

Note that we can easily verify that all conditions in \ref{Aas} and \ref{Aan} hold. We can also verify thanks to the exact formula of $\Sigma^*$, that this matrix is positive definite for the different values of $\rho$ used. This is a condition needed in Theorem \ref{aNormPoly}. \\
For any random estimator, constructing confidence intervals is important to assess the error in the estimation. For PR estimator, confidence interval can be obtained in one simulation run after estimating matrices $\Sigma^*$ and $A$ and hence the asymptotic variance matrix $V$. Figure \ref{Vest} shows the convergence, in the case $\rho = 0$, of the estimator of $V_n = A_n^{-1} \Sigma_n (A_n^{-1})^\intercal$  where $A_n$ and $\Sigma_n$ are as introduced in Proposition \ref{estVarJac}.
\begin{figure}[H]
    \centering
    \begin{subfigure}[h]{0.45\textwidth}
        \includegraphics[width=\textwidth]{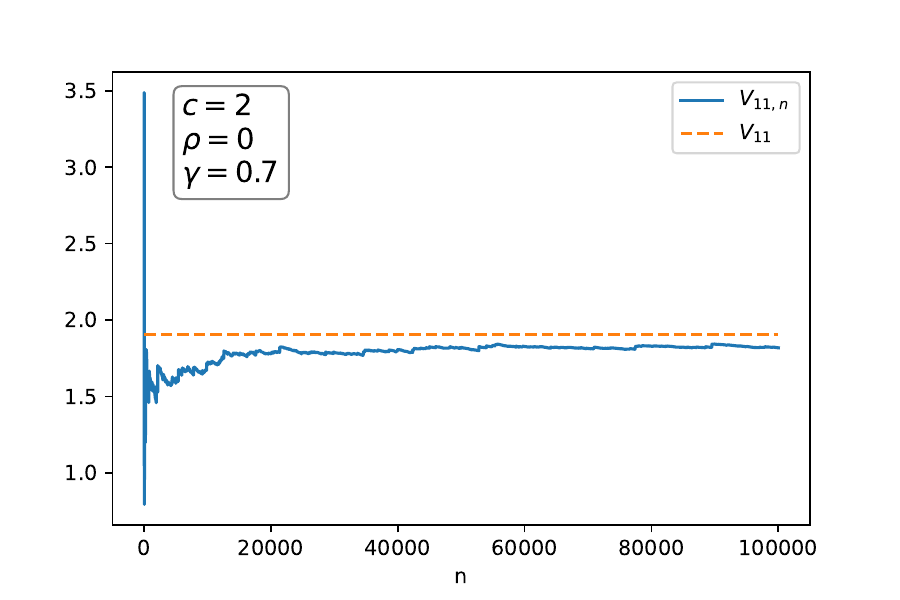}
    \end{subfigure}
    \begin{subfigure}[h]{0.45\textwidth}
        \includegraphics[width=\textwidth]{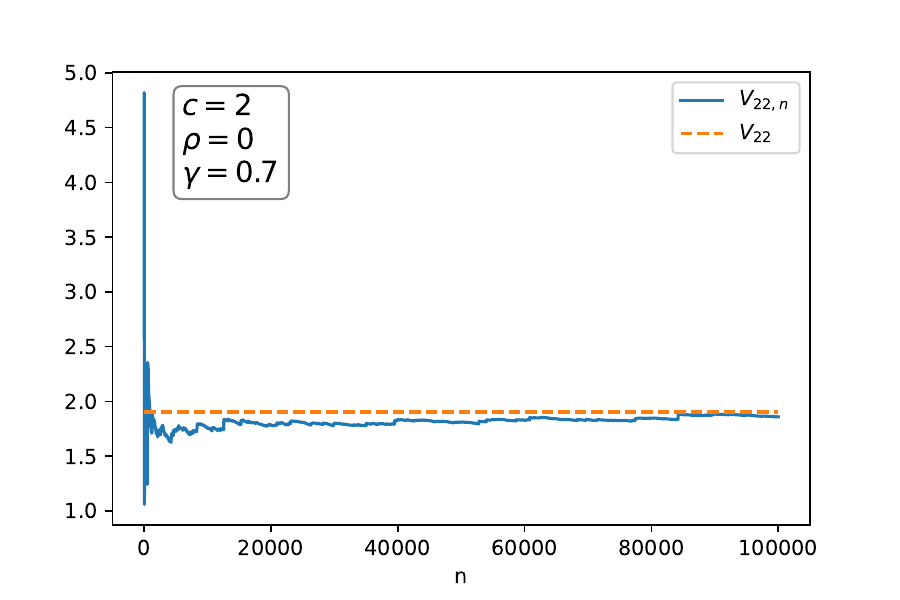}
    \end{subfigure}
    \caption{Convergence of the estimator $V_n$.}
    \label{Vest}
\end{figure}
In the following table, we give the estimated confidence interval for PR estimator with a confidence coefficient of $95\%$:
\begingroup
\setlength{\tabcolsep}{10pt} 
\renewcommand{\arraystretch}{1.5}
\begin{table}[H]
\caption{Confidence intervals for PR estimators.}
\centering
\begin{tabular}{|c|c|c|}
\hline
$\rho$ & CI for $m_1^*$ & CI for $m_2^*$\\
\hline
$-0.5$ & $[0.3772, 0.4047]$ & $[0.3679, 0.3949]$\\
$0$ & $[0.4962, 0.5259]$ & $[0.4912, 0.5213]$ \\
$0.5$ & $[0.6194, 0.6629]$ & $[0.6203, 0.6665]$\\
\hline
\end{tabular}
\end{table}
\endgroup
As for RM estimators, it is difficult to estimate the asymptotic covariance matrix due to its complexity. In order to visualize the normal behaviour of these estimators, we give the empirical probability density function (EPDF) in both cases $\gamma = 1$ and $\gamma = 0.7$. To this end,  we use again a number of steps $n = 100000$ and we repeat the procedure $N=10000$  times. We restrict our attention to the case $\rho =0$.
Figure \ref{ANFig1} shows that $D_{n,i} = \sqrt{n ^\gamma}(m_{n,i}-m_i^*), i \in \{1, 2\}$ are very close to a normal distribution.
\begin{figure}[H]
    \centering
    \begin{subfigure}[h]{0.45\textwidth}
        \includegraphics[width=\textwidth]{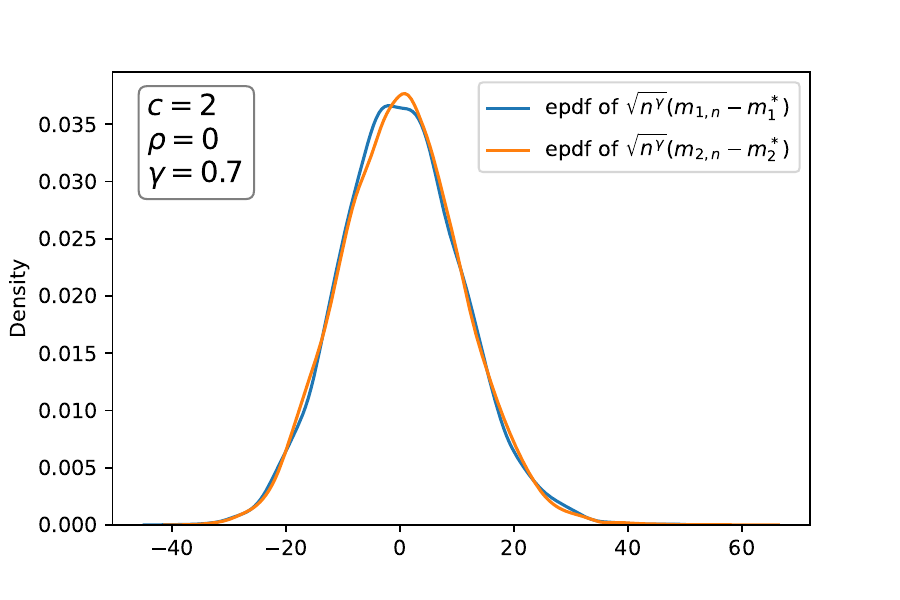}
    \end{subfigure}
    \begin{subfigure}[h]{0.45\textwidth}
        \includegraphics[width=\textwidth]{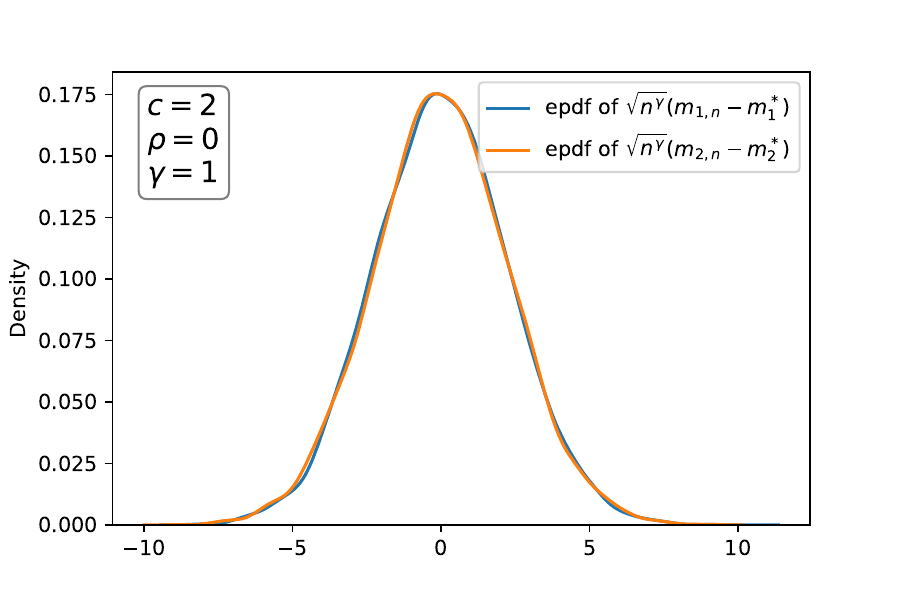}
    \end{subfigure}
    \caption{Empirical cumulative density function of $m_n - m^*$.}
    \label{ANFig1}
\end{figure}
In order to appreciate the quality of convergence of RM estimators, we also give the empirical cumulative density function (ECDF) of the error $m_n - m^*$. 
\begin{figure}[H]
    \centering
    \begin{subfigure}[h]{0.45\textwidth}
        \includegraphics[width=\textwidth]{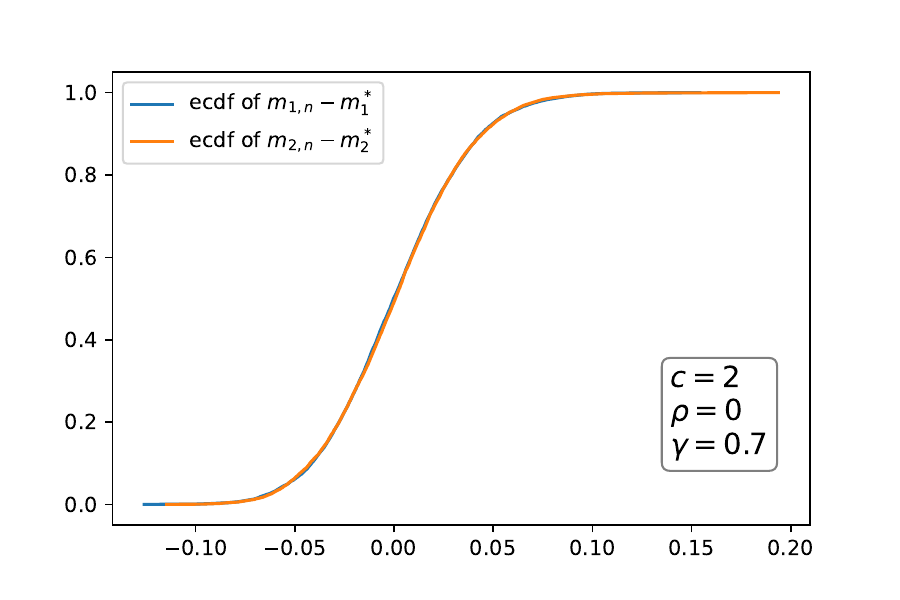}
    \end{subfigure}
    \begin{subfigure}[h]{0.45\textwidth}
        \includegraphics[width=\textwidth]{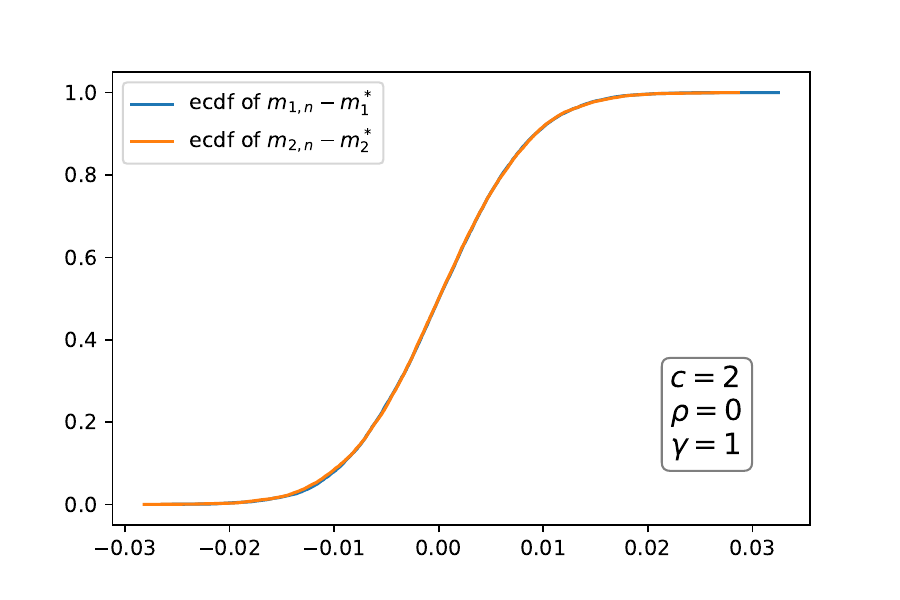}
    \end{subfigure}
    \caption{Empirical cumulative density function of $m_n - m^*$.}
    \label{ECDF}
\end{figure}
From the two figures above, the width of the $90\%$ confidence interval of the RM estimator for the case $\gamma =0.7$ is approximately $8\%$ and for the case $\gamma = 1$ is roughly $2\%$.
\subsection{Second example}
\label{SecondExample}
As a second example, we will consider consider the following loss function used in \cite{armenti}:
$$l(x_1,..., x_d) = \sum_{i=1}^d x_i + \frac{1}{2} \sum_{i=1}^d (x_i^+)^2 + \alpha \sum_{i < j} x_i^+ x_j^+.$$
\subsubsection{First case: Gaussian distribution and $d=2$}
We start by a simple case where we fix $d=2$ and use standard two dimensional Gaussian distribution for the loss vector $X$. We take $K=[0, 2]^3$, $n = 10^5$, $t = 10$, $\alpha = 1$ and  $c = 6$. Again, we compare RM and PR estimators for different values of $\rho$. {Note that  all assumptions can be easily  verified except the one that requires the matrix $cA +I/2$ to be a Hurwitz matrix. However, the PR estimator solve this issue. } \\
The following figure \ref{AsFig2} allows us to draw the same conclusions as in the previous example: RM estimator with $\gamma = 1$ and PR estimator are better than RM estimator with $\gamma = 0.7$. RM estimator with $\gamma = 0.7$ is noisy and one can remediate to this by choosing a smaller value of $c$ as we did in the first example.
\begin{figure}[H]
    \centering
    \begin{subfigure}[h]{0.4\textwidth}
        \includegraphics[width=\textwidth]{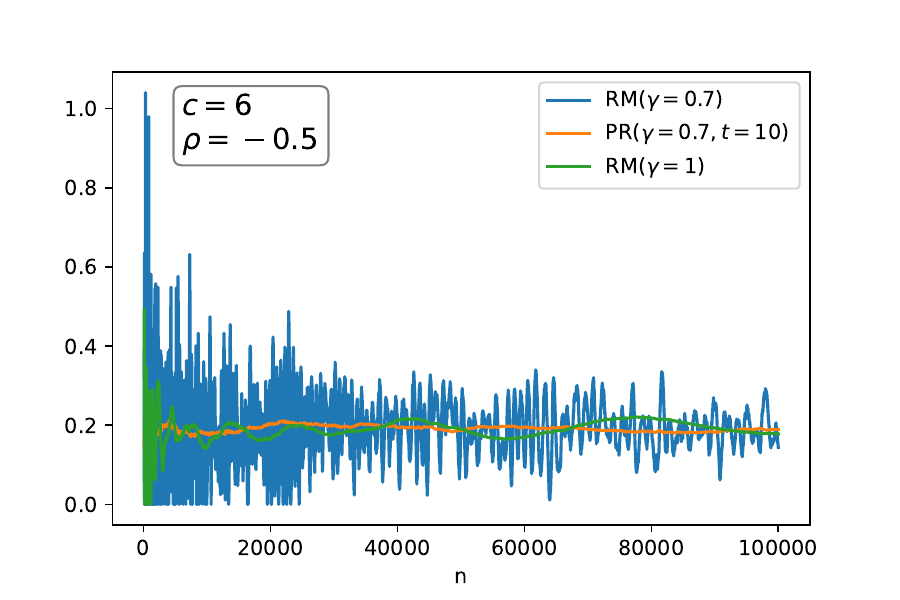}
    \end{subfigure}
    \begin{subfigure}[h]{0.4\textwidth}
        \includegraphics[width=\textwidth]{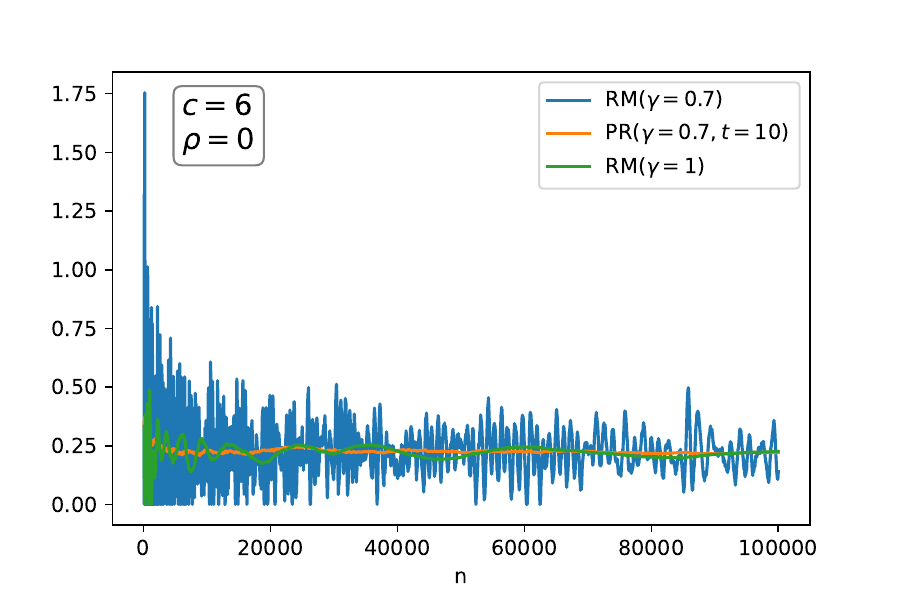}
    \end{subfigure}
    \begin{subfigure}[h]{0.4\textwidth}
        \includegraphics[width=\textwidth]{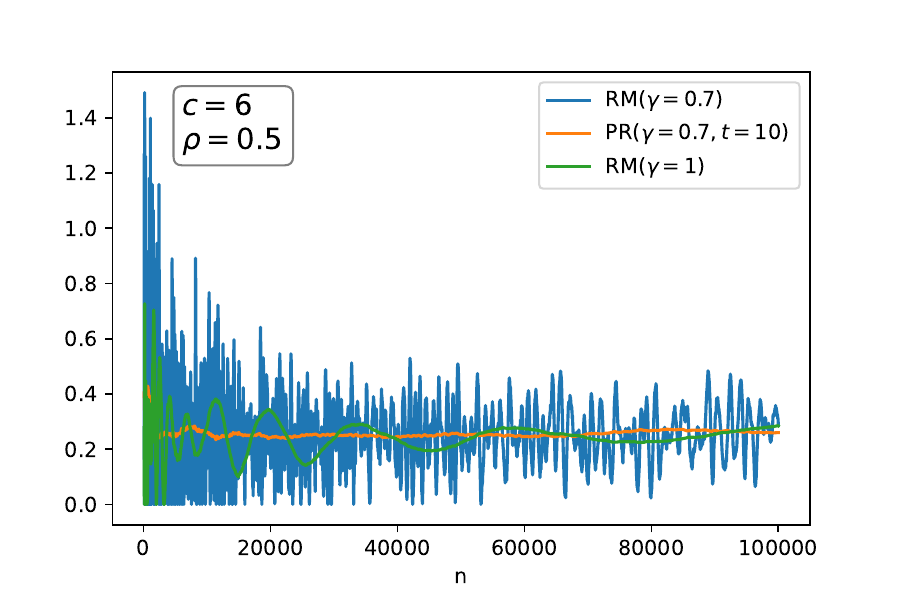}
    \end{subfigure}
    \caption{Consistency of RM/PR estimators with for different values of $\rho$.}
    \label{AsFig2}
\end{figure}
In order to assess the accuracy of our PR estimator, we give the confidence interval with a $95\%$ confidence coefficient (CI), using the estimators of Proposition \ref{estVarJac}. {Since $m_1^* =m_2^*$ in this example, we only present in Table \ref{tablePR} the results for the estimator $\hat{m}_1^*$ of $m_1^*$.}
\begingroup
\setlength{\tabcolsep}{10pt} 
\renewcommand{\arraystretch}{1.5}
\begin{table}[H]
\caption{Confidence intervals for PR estimators.}
\centering
\begin{tabular}{|c|c|c|}
\hline
$\rho$ & $\hat{m}_1^*$ & CI  \\
\hline
$-0.5$ & $0.188$ & $[0.1790, 0.2089]$ \\
$0$ &$0.21$ & $[0.1963, 0.2303]$  \\
$0.5$ &$0.25$ & $[0.2415, 0.2769]$ \\
\hline
\end{tabular}
\label{tablePR}
\end{table}
\endgroup
For RM estimator, we plotted the EPDF of $D_{n,i}, i\in \{1, 2\}$ as well as the ECDF of the error $m_n - m^*$ for the case $\rho = 0$ and $\gamma = 1$. These figures shows that the length of the confidence interval of 90\%  in the case $\gamma = 0.7$ is much higher that in the case $\gamma = 1$ (approx $0.2$ against $0.04$).
\begin{figure}[H]
    \centering
    \begin{subfigure}[h]{0.45\textwidth}
        \includegraphics[width=\textwidth]{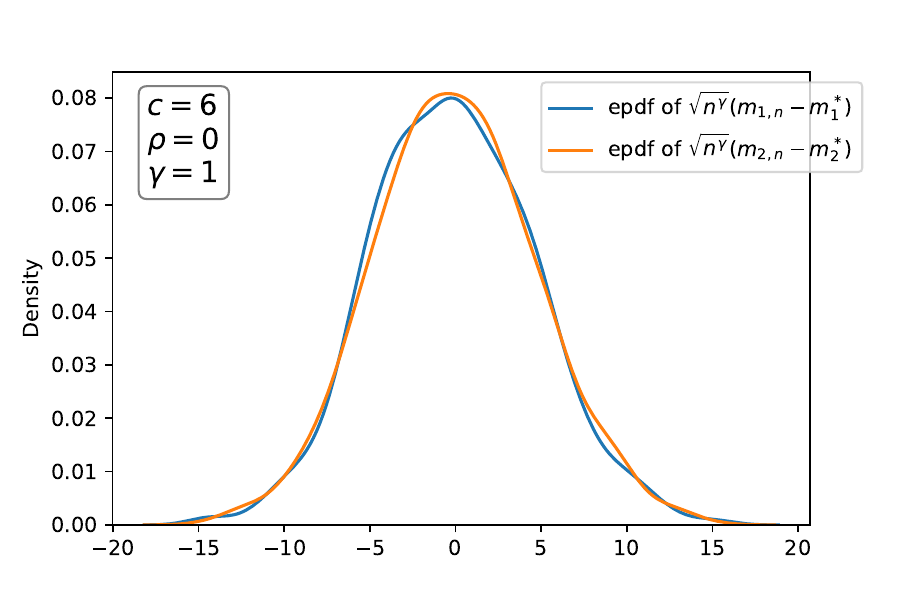}
    \end{subfigure}
    \caption{Empirical cumulative density function of $m_n - m^*$.}
    \label{ANFig2}
\end{figure} 
\begin{figure}[H]
    \centering
    \begin{subfigure}[h]{0.45\textwidth}
        \includegraphics[width=\textwidth]{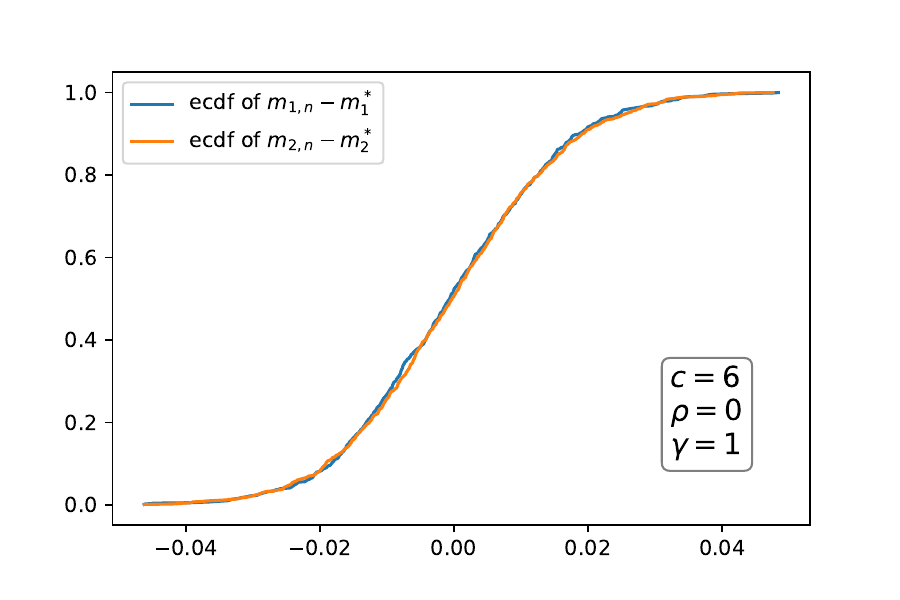}
    \end{subfigure}
    \caption{Empirical cumulative density function of $m_n - m^*$.}
    \label{ECDF2}
\end{figure}
\subsubsection{Second case: Compound Poisson Distribution and higher dimensions}
\label{secondCase}
In this section, we propose to use compound Poisson processes to model the loss vector $X$. The scope of application of compound Poisson processes is very wide. It ranges from statistical physics and biology to financial mathematics. In biology, they are used to study dynamics of populations. In the modern financial modeling, compound Poisson processes are used to describe dynamics of risk factors such as interest rates (see for instance \cite{compoundRates}), foreign exchange rates and option pricing (see \cite{catastrophe}). In actuarial science, compound processes are extensively used to model claims sizes and to compute the ruin probability, i.e. the probability that the initial reserves increased by premiums received from clients and decreased by their claims, drops below zero.\\
More precisely, given a final time $T$, we consider a multivariate Poisson random vector $N^T=(N_1^T,...,N_{d}^T)$, where each $N_i^T \sim \mathcal{P}(\lambda_i T)$ and the loss corresponding to the $i^{\mathrm{th}}$ component is $X_i = \sum_{k=1}^{N_i^T} G_i^k$ and $(G_i^k)_k$ is an i.i.d sequence representing the jump sizes and independent of $N_i^T$. We will take two examples for the distribution of the jumps sizes: One with a Gaussian distribution and another one with an exponential one. The correlation between the different components of $X$ will be done through the correlations between components of $N^T$.
In what follows, we detail the method of generating a multivariate Poisson random vector, $N = (N_1,...,N_d)$ with a vector of corresponding intensities $(\lambda_1,...,\lambda_d)$. To do so, we will use a method that is based on the Gaussian vectors. More precisely, denote $\eta = (\eta_1,...,\eta_d)$ to be a Gaussian random vector having a centered normal distribution with correlation matrix $R = (\rho_{kl})$ and $\Phi$ to be the standard normal cdf. Then, the random vector $\xi = (\Phi(\eta_1),...,\Phi(\eta_d))$ has a multivariate distribution with standard uniform marginal distributions. Let $P_\lambda(x)  = \sum_{j=0}^{[x]}(\lambda^j/j!)e^{-\lambda}$ be the cdf of the Poisson distribution with parameter $\lambda$. Now, consider the vector $\zeta=(\zeta_1,...,\zeta_d)$ where $\zeta_k = P^{-1}_{\lambda_k}(\Phi(\eta_k)), k=1,...,d$. $\zeta$ has therefore Poisson marginal distributions with intensities $(\lambda_1,...,\lambda_d)$. We can express the correlation coefficient  $\rho_{kl}^*=\mathrm{corr}(\zeta_k, \zeta_l)$ as a function of $\rho_{kl} = \mathrm{corr}(\eta_k, \eta_l)$: 
\begin{align*}
\rho_{kl}^* &= \frac{E(\zeta_k \zeta_l)-E(\zeta_k)E(\zeta_l)}{\sigma(\zeta_k)\sigma(\zeta_l)}\\
&= \frac{E(\zeta_k \zeta_l)-\lambda_k \lambda_l}{\sqrt{\lambda_k \lambda_l}}.
\end{align*}
We need to express the expectation $E(\zeta_k \zeta_l)$ as a function of $\rho_{kl}$. We have:
\begin{align*}
E(\zeta_k \zeta_l) &= E\left[P^{-1}_{\lambda_k}(\Phi(\eta_k))P^{-1}_{\lambda_l}(\Phi(\eta_l)) \right]\\
&= \sum_{m=1}^\infty \sum_{n=1}^\infty mn~P(\zeta_k = m, \zeta_l = n)\\
&= \sum_{m=1}^\infty \sum_{n=1}^\infty mn~P(u_{m-1}^{k} \le \Phi(\eta_k) \le u_m^{k}, u_{n-1}^{l} \le \Phi(\eta_l) \le u_n^{l}),
\end{align*}
where $u_j^i = P_{\lambda_i}(j), i,j=1,...,d$. It remains to explicit the probabilities in the last equality. If we denote $\Phi_2(\cdot, \cdot, \rho_{kl})$ the bivariate Normal distribution function, we get finally that,
\begin{align*}
Z_{mn}(\rho_{kl}) &\vcentcolon= P(u_{m-1}^{k} \le \Phi(\eta_k) \le u_m^{k}, u_{n-1}^{l} \le \Phi(\eta_l) \le u_n^{l})\\
&=\Phi_2(A_m^k, B_n^l, \rho_{kl})-\Phi_2(A_{m-1}^k, B_n^l, \rho_{kl})-\Phi_2(A_m^k, B_{n-1}^l, \rho_{kl})+\Phi_2(A_{m-1}^k, B_{n-1}^l, \rho_{kl})
\end{align*}
where $A_m^k = \Phi^{-1}(P_{\lambda_k}(m))$ and $B_n^l = \Phi^{-1}(P_{\lambda_l}(n))$. As a conclusion, we obtain,
\begin{equation}\label{calib}
\sum_{m=1}^\infty \sum_{n=1}^\infty mn Z_{mn}(\rho_{kl}) = \lambda_k \lambda_l + \rho_{kl}^*\sqrt{\lambda_k \lambda_l}.
\end{equation}
The equation \eqref{calib} gives an implicit relation between $\rho_{kl}^*$ and $\rho_{kl}$. It also involves two infinite sums which makes it hard to solve. In practice, one needs to truncate this sum and choose some appropriate upper-limits $M^*$ and $N^*$. We are then able to compute the elements of the correlation matrix $\rho$ of the Gaussian vector given the correlation matrix $\rho^*$ of the vector $N$. However, there is a problem of sufficient conditions for a given positive semi-definite matrix to be a correlation matrix of a multivariate Poisson random vector. This issue is tackled in \cite{poisson} where it is shown that each $\rho_{kl}^*$ has to be in a certain range,
\begin{equation}\label{constRho}
-1 < \rho_{kl}^{\min} \le \rho_{kl}^* \le \rho_{kl}^{\max} \le 1.
\end{equation}
\begin{algorithm}[H]
\caption{Algorithm for generating a sample of $X$ with Compound Poisson Distribution}\label{algo}
\KwInput{$(\lambda_i)_{i=1,...,d}$, $(\sigma_i)_{i=1,...,d}$, $T$ , and $\rho^*$ correlation matrix of $N_T$\;}
\KwEnsure{For each $k > l$, $\rho_{kl}^*$ verifies the inequality in \eqref{constRho}\;}
Solve the equation \eqref{calib} to find $\rho$ the correlation of the Gaussian vector $\eta = (\eta_1,...,\eta_d)$\;
Generate a sample of Gaussian vector $\eta$ with correlation matrix $\rho$ and for $i=1,...,d$\;
\For{$i=1,...,d$}{
Set $N_i^T = P^{-1}_{\lambda_i T}(\Phi(\eta_i))$\;
Generate a i.i.d sample of $G_i^k$ of size $N_i^T$\;
Set $X_i = \sum_{k=1} G_i^k$.}
\KwOutput{X\;}
\end{algorithm}
The following figure shows the covariance matrix for the loss vector $X$ of dimension $d=10$ obtained by generating a random correlation matrix $(\rho_{kl})$ and using a Gaussian distribution for the jump sizes, i.e. $G_i^k \sim \mathcal{N}(1, 1)$. The intensity vector was taken uniformly in $[1,3]^{10}$.
\begin{figure}[H]
    \centering
    \includegraphics[scale = 0.8]{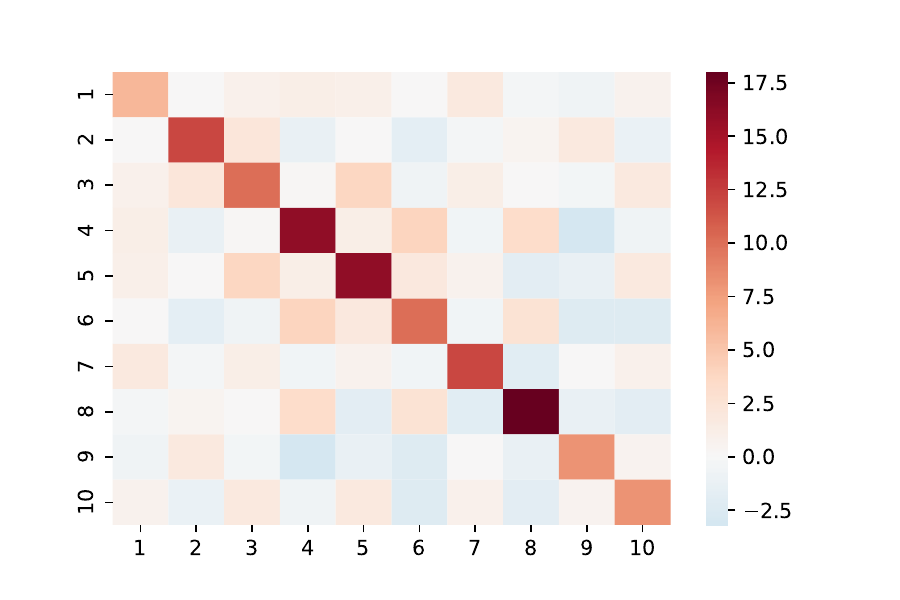}
    \caption{Correlation matrix of the vector loss $X$ with Gaussian jump sizes with means and variances equal to $1$.}
    \label{correlation}
\end{figure}
Setting $ K = [-20, 20]^{10}\times[0, 20]$, the averaging parameter $t=10$ and $c=6$, $\gamma = 0.7$ and the number of steps $n=100000$, we obtain the following optimal allocations for both cases $\alpha = 0$ and $\alpha = 1$.
\begin{figure}[H]
    \centering
    \includegraphics[scale=0.8]{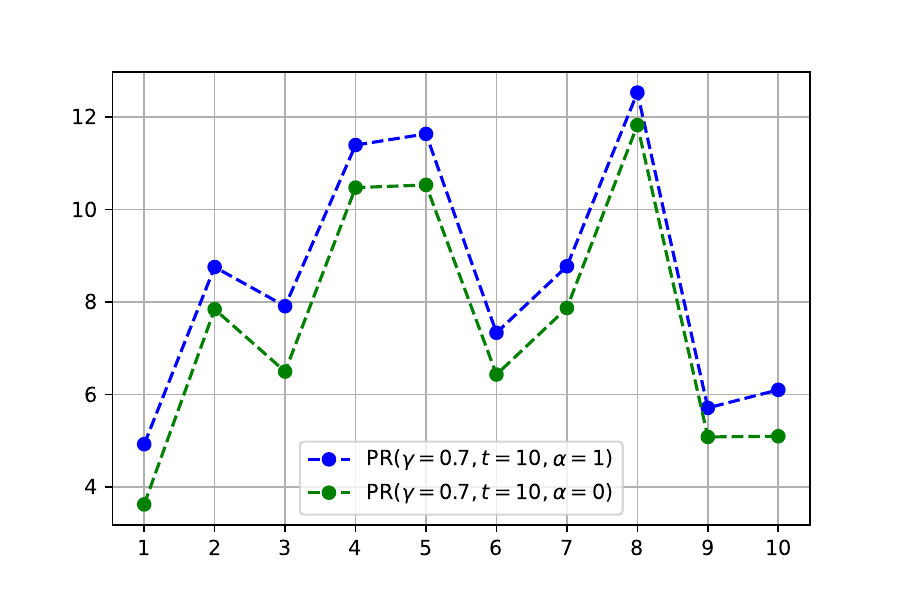}
    \caption{PR estimators for optimal allocations.}
    \label{optAlloc10}
\end{figure}
The above figure shows that there are components with the same optimal allocations for the case $\alpha = 0$. This is something we expect to see, since with $\alpha = 0$, correlations between components are not involved, so components with the same variance should have the same optimal allocations. This is the case for components $4-5$ and components $9-10$. However, once $\alpha$ is taken non null, we see that the same components have no longer the same optimal allocations. For instance, component $10$ has higher optimal allocation than $9$ when $\alpha = 1$. This could be explained by the fact that component $10$ is more correlated with other components that have high variances, such as components $4$ and $5$, than component $9$. We now consider an exponential distribution for jump sizes as a second example, i.e. $G_i^k \sim \mathcal{E}(a_i)$. The parameters $a_i$ were generated randomly in $[0.2, 1.2]$. As for the other paramaters in this example, we took again $K=[-20, 20]^{10}\times[0, 20]$, $c=6$, $t=10$ and $\gamma = 0.7$. Covariance matrix of the loss vector $X$ in this case and estimators of the optimal allocations obtained through PR algorithms with a number of steps $n=100000$ together with corresponding confidence intervals are given in the following figures.
\begin{figure}[H]
\centering
    \includegraphics[scale=0.8]{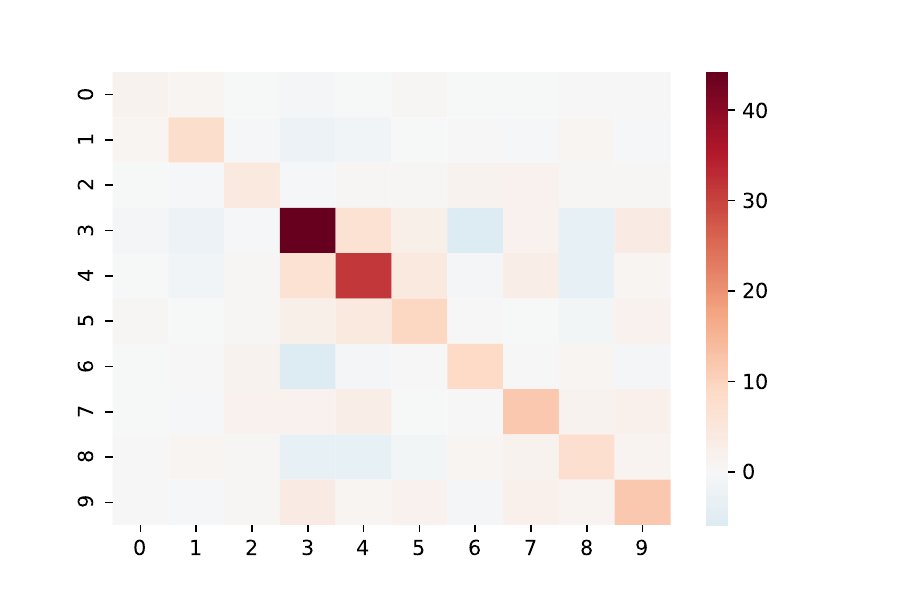}
    \caption{Correlation matrix of the vector loss $X$ with exponential distribution for the law of jump sizes.}
\end{figure}
\begin{figure}[H]
    \centering
    \begin{subfigure}[h]{0.6\textwidth}
        \includegraphics[width=\textwidth]{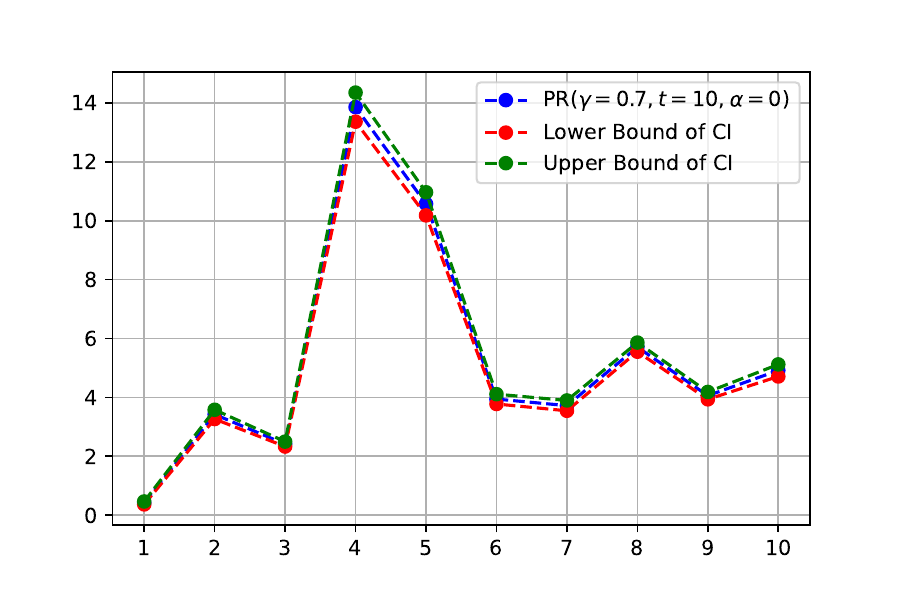}
    \end{subfigure}
    \begin{subfigure}[h]{0.6\textwidth}
        \includegraphics[width=\textwidth]{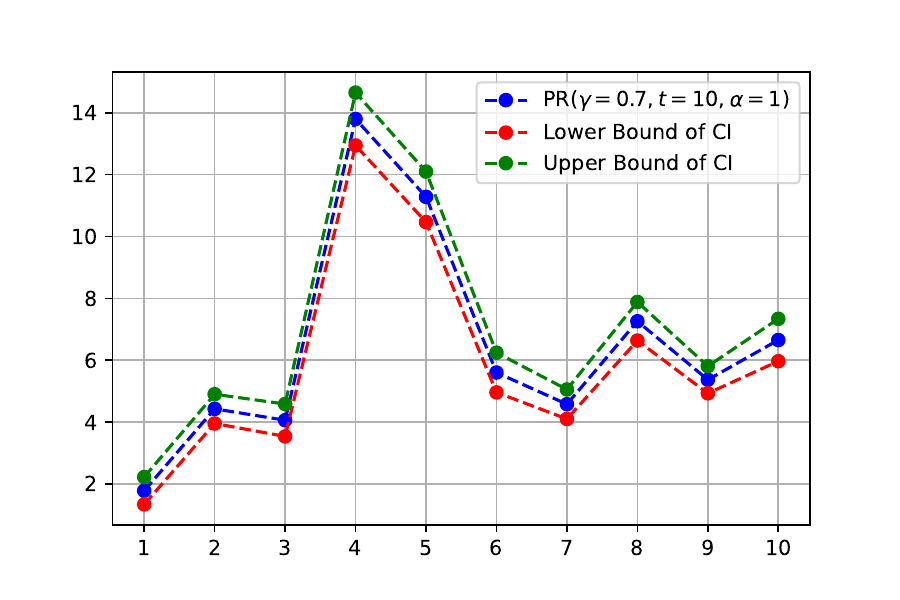}
    \end{subfigure}
    \caption{PR estimators for optimal allocations together with bounds of CI.}
\end{figure}
\section{Appendix}
\label{6}
\subsection{ODE method and related concepts}
\label{ODEMethod}

In this section, we recall some key concepts of the stability of an ODE $\dot{z} = h(z)$.
\subsubsection{Concepts of stability of an ODE}
A state $z^*$ is an equilibrium of the ODE if $h(z^*)=0$. To describe the behaviour of the system around the equilibrium, a number of stability concepts are needed. Let us first introduce the basic concepts of stability. To alleviate the notations, we will take $0$ as an equilibrium state.
\begin{Definition}
The equilibrium $z^*=0$ is said to be stable, if for any $R >0$, there exists $r>0$ such that if $||z(0)|| < r$, then $||z(t)|| < R$ for all $t\ge 0$. Otherwise, the equilibrium is unstable.
\end{Definition}
Essentially, this means, the system can be kept arbitrarily close to the origin by starting sufficiently close to it. This is also know as \textit{Lyapunov} stability. In some applications, Lyapunov stability is not enough and the concept of asymptotic stability is needed.
\begin{Definition}
An equilibrium point $z^*=0$ is asymptotically stable if it is stable, and if in addition, there exists some $r > 0$ such that $||z(0)|| < r$ implies that $z(t)\rightarrow 0$ as $t \rightarrow \infty$. The ball $B_r$ is called a domain of attraction of the equilibrium point.\\
If asymptotic stability holds for all initial states, the equilibrium is said to be globally asymptotically stable.
\end{Definition}

\subsubsection{Lyapunov function and invariant set}
The first property that need to be verified by a Lyapunov function is \textit{positive definiteness}.
\begin{Definition}
A scalar continuous function $V(z)$ is said to be locally positive definite if $V(0)=0$ and in around $0$, we have, $z \neq 0 \Rightarrow V(z) > 0$.\\
If the above property holds over the whole state space, then $V(z)$ is said to be globally positive definite.
\end{Definition}
The above definition implies that the function $V$ has a unique minimum at the origin $0$. 
Next, we define the \q{derivative of V} with respect to time along the system trajectory. Assuming that $V$ is differentiable, this derivative is defined as:
$$\dot{V}(z) = \frac{dV(z)}{dt} = \nabla V, \dot{z}= \nabla V\cdot h(z).$$
\begin{Definition}[Lyapunov function]
Let $V$ be a positive definite function and continuously differentiable. If its time derivative along any state trajectory is negative semi-definite, i.e., $$\dot{V}(z) = \nabla V_z\cdot h(z) \le 0,~~ \forall z,$$
then $V$ is said to be a Lyapunov function for the system.
\end{Definition}
It is possible to draw conclusions on asymptotic stability, with the help of the invariant set theorems introduced by LaSalle. The central concept in these theorems is that of invariant set, a generalization of the concept of equilibrium point.  

\begin{Definition}
Let $z(\cdot)$ be a solution of some ODE. A set $G$ is said to be an invariant set for this ODE if $z(0) \in G$ implies that $z(t) \in G, ~\forall t \ge 0$.
\end{Definition}
For instance, the singleton $\{z^*\}$ where $z^*$ is an equilibrium point is an invariant set. Its domain of attraction is also an invariant set. One other trivial invariant set is the whole state-space, $\underset{z_0}{\cup}\{z(t),~t>0,~z(0)=z_0\}$. \\
Informally, the Lasalle invariance theorem allows the identification of an ODE's limit set as a subset of  the largest invariant set contains in the set $\{z  ; \; W(z) =0\}$, with $W$ some well-chosen function (most of the time $W=\dot{V}$). 
We recall below Barbalat's Lemma (see e.g. Lemma 3 in \cite{fischer2013lasalle}), with is central to proving  LaSalle's invariance theorem.
\begin{Lemma}[Barbalat's Lemma]
\label{BarbalatLemma}
Let $\phi : \R^+ \to \R$ be a uniformly continuous function. Suppose that $\lim_{t\to \infty} \int_0^t \phi(s) ds <\infty$. Then, $\lim_{t\to \infty} \phi(t) = 0$. 
\end{Lemma}

\subsection{ Strong and weak convergence of constrained stochastic algorithms}
\label{AppendixKushner}

Let  $g : \R^n \to \R^n$ be a continuous function, and let $z^*\in \R^n$ be an equilibrium point of $h$, i.e. $h(z^*)= 0$. \\
The constrained stochastic approximation algorithm is defined by 
\begin{equation}
Z_{n+1} = \Pi_K[Z_n + \epsilon Y_n ], 
\end{equation}
with $(Y_n)$ a sequence of $\R^n$ random variables such that $E[Y_n | \mathcal{F}_n]= g(Z_n) $, $ \Pi_K$ the projection operator on the constraint set $K = \{ z ; \;  a_i \leq z_i \leq b_i\} \subset \R^n$ and $(\epsilon_n)$ a sequence decreasing to zero, such that $\sum_{n=0}^\infty \epsilon_n = \infty$.  
\paragraph{Continuous time interpolation}  let $t_0 = 0$ and for $n\geq 1$,  $t_n = \sum_{i=0}^{n-1} \epsilon_i$. The piecewise constant continuous-time interpolation of the sequence $(Z_n)$ is defined  by 
\begin{equation*}
Z^0(t) = Z_n, \quad \forall \; t \in [t_n, t_{n+1}[. 
\end{equation*}
The shifted process is defined by
\begin{equation}
\label{interpolation}
Z^n(t) = Z^0(t_n +t), \quad \forall t\geq 0. 
\end{equation}
Similarly, let 
\begin{equation*}
U^0(t) = \frac{Z_n-z^*}{\sqrt{\epsilon_n}}, \quad \forall \; t \in [t_n, t_{n+1}[, 
\end{equation*}
and 
\begin{equation}
\label{interpolation1}
U^n(t) = U^0(t_n +t), \quad \forall t\geq 0. 
\end{equation}
We recall below the assumption and statement of Theorem 5.2.1 (p.127) and 10.2.1 (p.329) in \cite{kushner}. Note that the framework presented here is restricted to the case where $E[Y_n|\mathcal{F}_n ] = g(Z_n)$ ($g$ does not depend on $n$). Hence,  some assumptions  are either trivial and not rewritten here or simplified. 
\paragraph{Strong convergence}
The assumption below can be found in Section 5.2 of \cite{kushner}.
\begin{Assumption}
\label{AssumptionKushner1}
\begin{enumerate}
\item \label{Kush1A1} (A2.1) $\sup E[|Y_n^2|] <\infty$.
\item \label{Kush1A3} (A2.3) $g$ is continuous. 
\item \label{Kush1A4} (A2.4) $\sum_i \epsilon_i^2 <\infty$. 
\end{enumerate}
\end{Assumption}
\begin{Theorem}[Theorem 5.2.1 in \cite{kushner} (partial)]
\label{ThKusher1}
Let $\{Z^n(\cdot)\}$ be the sequence of continuous-time processes defined by \eqref{interpolation} and let Assumption \eqref{AssumptionKushner1} holds. \\
There is a set $N$ of probability zero such that for $\omega \notin N$, the family of function $\{Z^n(\omega,\cdot)\}$ is equicontinuous. Furthermore, if  $z(\cdot)$ is a limit of a convergent subsequence,  then the function is solution of the projected ODE
\begin{equation*}
\dot{z} = g(z) + C(z), \quad C(z(t)) \in -\mathcal{C}(z(t)), 
\end{equation*}
and $\{Z_n(\omega)\}$ converge to some limit  set of the ODE in K. 
\end{Theorem}
\paragraph{Weak convergence} 
The assumptions below can be found in Section 10.2 and 10.4 of \cite{kushner}:
\begin{Assumption}
\label{AssumptionKushner2}
${}$
\begin{enumerate}
\item \label{A1} (A2.0) Either $\epsilon_n = \frac{1}{n}$ or $\mu_n = o(\epsilon_n)$, for  $\sqrt{\frac{\epsilon_n}{\epsilon_{n+1}}} = 1 + \mu_n$.
\item \label{A2} (A2.1) and  (A4.2)  $(Y_n\mathbf{1}_{|Z_n - z^*|\leq \rho})$ is uniformly integrable for some $\rho >0$. 
\item \label{A3}(A2.2)  Let  $(Z^n(\cdot))$ be the sequence of continuous time processes defined by \eqref{interpolation}. Then $(Z^n(\cdot))$ converges weakly to $z^*$.  
\item \label{A4} (A2.4) $g$ is continuously differentiable. 
\item \label{A5} (A2.6 and A4.4) The Jacobian matrix $A'$ of $g$ at $z^*$ is a Hurwitz matrix. If $\epsilon_n = \frac{1}{n}$,  then $A' + \frac{I}{2}$ is also a Hurwitz matrix. 
\item \label{A6}(A2.7)  For some $p > 0$ and $\rho > 0, \underset{|z-z^*| \le \rho}{\sup} m^{2+p}(z) < \infty$. Furthermore, 
$$E[(Y_n-g(Z_n))(Y_n- g(Z_n))^\intercal\mathbf{1}_{|Z_n - z^*|\leq \rho}]\to \Sigma_1$$ in probability, with $\Sigma_1$ some non-negative definite matrix. 
\item \label{A9} (A4.5) 
\begin{itemize}
\item Case $\epsilon_n = \frac{1}{n}$: let $P'$ be the unique symmetric and positive definite solution of the Lyapunov equation $(A')^\intercal P' + P'A' = -I$, and $\lambda$ the largest positive number such that $I \geq \lambda P'$. Then $\lambda>1$. 
\item Case $\mu_n = o(\epsilon_n)$: Let $m(t)$ be the smallest integer $n$ such that $t_n \leq t < t_{n+1}$, with $m(t)=0$ for $t\leq 0$. Then for all $T>0$,  $\liminf_n \min_{m(t_n-T)\leq i \leq n} \frac{\epsilon_n}{\epsilon_i } = 1$. 
\end{itemize}
\end{enumerate}
\end{Assumption}
\begin{Theorem}[Theorem 10.2.1 in \cite{kushner} (partial)]
\label{ThKushner2}
 Assume that the assumptions \ref{AssumptionKushner1} and \ref{AssumptionKushner2} hold, and  let $U^n(\cdot)$ be the process defined by  \eqref{interpolation1}. \\
If $\mu_n =o(\epsilon_n)$, then the sequence of processes $(U^n(\cdot))$ converges weakly in the Skorohod space $D^n(\R^+)$ to the stationary process $U$ defined for all $t\geq 0$ by: 
\begin{equation}
\label{Udef1}
U(t) =  \int_{-\infty}^t e^{A'(t-s)}dW_s. 
\end{equation}
If $\epsilon_n =\frac{1}{n}$, then the sequence of processes $(U^n(\cdot))$ converges weakly to the stationary process $U$ defined for all $t\geq 0$ by: 
\begin{equation}
\label{Udef2}
U(t) =  \int_{-\infty}^t e^{(A'+\frac{I}{2})(t-s)}dW_s, 
\end{equation}
with in both cases $W: \R \to \R^n$ a Wiener process of covariance matrix $\Sigma_1$. 
\end{Theorem}
\begin{proof}
By Assumption \ref{AssumptionKushner1} and \ref{AssumptionKushner2}, the assumption (A2.0)-(A2.7)  p.329 of Theorem 10.2.1 are straightforward, at the exception of assumption (A2.3). The latter condition is also satisfied as a consequence of Theorem 10.4.1 p.341 in \cite{kushner} since the assumptions (A4.1)-(A4.5) of this Theorem are also satisfied by Assumption \ref{AssumptionKushner1} (1), Theorem \ref{ThKusher1}, and Assumption \ref{AssumptionKushner2} (5) and (7). 
\end{proof}

\bibliographystyle{apacite}
\bibliography{biblio}  






\end{document}